\RequirePackage{fix-cm}

\documentclass[smallcondensed]{svjour3}     

\smartqed  
\usepackage{graphicx}

\usepackage{amsmath}
\usepackage{stix}
\usepackage{ulem}

\usepackage{bm}
\usepackage{color}
\usepackage{float}

\def\D{\Omega}
\def\T{\mathcal{T}}
\def\E{\widetilde{E}}

\newcommand{\cT}{\mathcal{T}}
\renewcommand{\div}{\mbox{\rm div\,}}

\newcommand{\eps}{\epsilon}

\newcommand{\Ome}{\Omega}
\newcommand{\p}{\partial}
\newcommand{\nab}{\nabla}
\def\esssupT{\underset{t\in [0,T]}{\mbox{\rm ess sup }}}

\begin{document}

\title{Error analysis of a fully discrete Morley finite element approximation for the Cahn-Hilliard equation}

\titlerunning{The Morley finite element approximation for the Cahn-Hilliard equation} 



\author{Yukun Li}
\institute{Yukun Li \at
              Department of Mathematics, The Ohio State University, Columbus \\
              Tel.: 865-456-9589\\
              \email{li.7907@osu.edu}
}

\date{Received: date / Accepted: date}

\maketitle

\begin{abstract}
This paper proposes and analyzes the Morley element method for the Cahn-Hilliard equation. It is a fourth order nonlinear singular perturbation equation arises from the binary alloy problem in materials science, and its limit is proved to approach the Hele-Shaw flow. If the $L^2(\Omega)$ error estimate is considered directly as in paper \cite{elliott1989nonconforming}, we can only prove that the error bound depends on the exponential function of $\frac{1}{\epsilon}$. Instead, this paper derives the error bound which depends on the polynomial function of $\frac{1}{\epsilon}$ by considering the discrete $H^{-1}$ error estimate first. There are two main difficulties in proving this polynomial dependence of the discrete $H^{-1}$ error estimate. Firstly, it is difficult to prove discrete energy law and discrete stability results due to the complex structure of the bilinear form of the Morley element discretization. This paper overcomes this difficulty by defining four types of discrete inverse Laplace operators and exploring the relations between these discrete inverse Laplace operators and continuous inverse Laplace operator. Each of these operators plays important roles, and their relations are crucial in proving the discrete energy law, discrete stability results and error estimates. Secondly, it is difficult to prove the discrete spectrum estimate in the Morley element space because the Morley element space intersects with the $C^1$ conforming finite element space but they are not contained in each other. Instead of proving this discrete spectrum estimate in the Morley element space, this paper proves a generalized coercivity result by exploring properties of the enriching operators and using the discrete spectrum estimate in its $C^1$ conforming relative finite element space, which can be obtained by using the spectrum estimate of the Cahn-Hilliard operator. The error estimate in this paper provides an approach to prove the convergence of the numerical interfaces of the Morley element method to the interface of the Hele-Shaw flow.
\end{abstract}

%

\section{Introduction}
Consider the following Cahn-Hilliard problem:
\begin{alignat}{2}
u_t +\Delta(\epsilon\Delta u -\frac{1}{\epsilon}f(u)) &=0  &&\quad \mbox{in } \Omega_T:=\Omega\times(0,T],\label{eq20170504_1}\\
\frac{\partial u}{\partial n}
=\frac{\partial}{\partial n}(\epsilon\Delta u-\frac{1}{\epsilon}f(u))
&=0 &&\quad \mbox{on } \partial\Omega_T:=\partial\Omega\times(0,T],
\label{eq20170504_2}\\
u &=u_0 &&\quad \mbox{in } \Omega\times\{t=0\},\label{eq20170504_3}
\end{alignat}
where $\Omega\subseteq \mathbf{R}^2$ is a bounded domain, $f(u)$ is the first derivative of a double well potential $F(u)$ which is defined below
\begin{equation}\label{eq20170504_5}
F(u)=\frac{1}{4}(u^2-1)^2.
\end{equation}

The Allen-Cahn equation \cite{Allen_Cahn79,Bartels_Muller_Ortner09,Chen94,Feng_Prohl03,feng2014finite,feng2017finite,Ilmanen93,kovacs2011finite}, which is a second order nonlinear parabolic equation, describes the phase separation process of a binary alloy when the temperature suddenly decreases, but the mass of each phase is not conserved. Compared with the Allen-Cahn equation, the Cahn-Hilliard equation \eqref{eq20170504_1} also arises from the phase transition problem in materials science, but it has the mass conservation property. Notice equation \eqref{eq20170504_1} differs from the original Cahn-Hilliard equation by scaling $\frac{t}{\epsilon}$ by $t$. The Cahn-Hilliard equation finds its applications in the areas of materials science, fluid mechanics, biology and so on, and the coupling of the Cahn-Hilliard equation and fluid flow is becoming more and more popular in industrial applications. The Cahn-Hilliard equation also serves as a building block for the phase field formulations of the moving interface problems, and the methodology can be applied to other phase field models.
It is also well known \cite{Alikakos94} that the Cahn-Hilliard equation 
\eqref{eq20170504_1} can be interpreted as the $H^{-1}$ gradient flow of the Cahn-Hilliard energy functional
\begin{align}\label{eq2.1}
J_\epsilon(v):= \int_\Omega \Bigl( \frac\eps2 |\nabla v|^2+ \frac{1}{\epsilon} F(v) \Bigr)\, dx.
\end{align}
Stoth proved that $u\rightarrow\pm1$ in the interior or exterior of interface $\Gamma_t$ for all $t\in[0,T]$ as $\epsilon\rightarrow0$ for the radially symmetric case \cite{Stoth96}, and Alikakos, Bates and Chen gave the proof for the general case \cite{Alikakos94}. 

Numerical approximations of the Cahn-Hilliard equation have been extensively studied during the last 30 years \cite{Aristotelous12,xu2016convex,du1991numerical,elliott1989nonconforming,cheng2017energy}. These papers consider the case when $\epsilon$ is a fixed, and the error bounds depend exponentially on $\frac{1}{\epsilon}$. Better than the exponential dependence on $\frac{1}{\epsilon}$, the polynomial dependence on $\frac{1}{\epsilon}$ is proved using conforming finite element (CG) method \cite{Feng_Prohl04,Feng_Prohl05} and discontinuous Galerkin (DG) method \cite{feng2016analysis,li2015numerical}. For the $C^1$ conforming finite elements for the fourth order problem, polynomials with high degree are required. To use lower order polynomials, one approach is to use macro-elements, where a given element is divided into a few smaller subelements and the lower order polynomial is used on each subelement. However, it is not widely used due to its complex formulation of finite element spaces. The other approach is to use nonconforming finite elements, and among the nonconforming finite elements, the Morley element has the least number of degrees of freedom on each element. Comparing with the mixed finite element method or the $C^1$ conforming finite element method, the computational cost of the Morley element is smaller, and this is extremely important especially for the phase field problems where the interaction length $\epsilon$, time step size $k$, and mesh size $h$ are all required to be chosen very small. The Morley element was first used in \cite{elliott1989nonconforming} to discretize the Cahn-Hilliard equation, but only the error estimates with exponential dependence of $\frac{1}{\epsilon}$ could be derived there.
In this paper, the polynomial dependence of $\frac{1}{\epsilon}$ is finally given using the Morley element.

The approach in this paper follows those used in \cite{feng2016analysis,Feng_Prohl04,Feng_Prohl05}, but the generalization to the Morley element method is nontrivial. In the mixed CG/DG formulation, different test functions can be chosen in two equations, but in the Morley element formulation, only one test function can be chosen. Because of this and the complex structure of the Morley element formulation, proving the discrete energy law and the discrete stability results become much more involved. It is also a challenge to prove the discrete spectrum estimate in the Morley element space from the spectrum estimate of the Cahn-Hilliard operator because the Morley element space has intersection with its $C^1$ conforming relative finite element space but they are not contained in each other. If the $L^2$ error estimate is considered directly, the generalized coercivity result in this paper or even the discrete spectrum estimate are not useful in proving the $L^2$ error estimate. To overcome these difficulties, there are three main techniques in this paper. First, based on the structure of the bilinear form of the Morley element formulation, this paper designs four discrete operators $\hat{\Delta}^{-1}_h, \widetilde{\Delta}^{-1}_h$, $\underline{\Delta}^{-1}_h$ and $\Delta^{-1}_h$, and proves the errors in different norms between these operators. Through these relations, by using $\widetilde{\Delta}^{-1}_h$ in the test function, and by using the other operators as bridges, we can prove the discrete energy law and some consequent discrete stability results. Each of these operators plays important roles in proving the main results. These operators and their properties might be applied to the analysis for the biharmonic equation. It also employs both the summation by part for time and integration by part for space techniques to handle the nonlinear term and then to establish the polynomial dependence of the $\|\cdot\|_{2,2,h}$ stability result, and only the exponential dependence can be obtained if these two techniques are not used simultaneously. Second, instead of proving the discrete spectrum estimate, this paper proves the generalized coercivity result which is sufficient to get the sharper error estimates. The key point is to use the enriching operator as a bridge between the nonconforming and conforming finite elements, and this idea may be extended to other phase field models. Third, if the discrete $L^2$ error estimate is considered directly, only the error estimates with exponential dependence on $\frac{1}{\epsilon}$ could be derived using the Gronwall's inequality as in \cite{elliott1989nonconforming}. This paper provides a possibility by considering the $H^{-1}$ error estimate first, and it explains how to utilize the discrete inverse Laplace operators and the generalized coercivity result to circumvent the Gronwall's inequality, and finally to prove the error estimate with polynomial dependence on $\frac{1}{\epsilon}$. 

The remainder of this paper is organized as follows. In section 2,  we introduce the standard function and Sobolev space notations, state a few a priori estimates of the solution, and cite some known results including properties of the inverse Laplapce operators, properties of enriching operator, generalized
discrete Gronwall's inequality and the spectrum estimate for the linearized Cahn-Hilliard operator; In section 3, and in the first two subsections, we introduce the Morley element formulation, define different kinds of discrete inverse Laplace operators and state their relations. Then in the last three subsections, we analyze the discrete energy law and the discrete stability results, derive the generalized coercivity result in the Morley element space, and finally prove the discrete $H^{-1} $ error estimate with polynomial dependence on $\frac{1}{\epsilon}$; In Section \ref{sec4}, numerical experiments are given to validate the theoretical results.

\section{Preliminaries}
In this section, we cite some known results about problem \eqref{eq20170504_1}--\eqref{eq20170504_5}, and they will be used in the following sections. These results can be proved under some general assumptions on the initial condition \cite{Chen94,Feng_Prohl04,Feng_Prohl05,feng2016analysis,li2015numerical}. 
Throughout this paper, $C$ denotes a generic positive constant, which may have different values at different occasions, is independent of interfacial length $\epsilon$, spacial size $h$, and time step size $k$. The following Sobolev notations are used in this paper, i.e., for any set $A$,
\begin{alignat*}{2}
\|v\|_{0,p,A}&=\bigg(\int_{A}|v|^pdx\bigg)^{1\slash p}\qquad &&1\le p<\infty,\\
\|v\|_{0,\infty,A}&=\underset{A}{\mbox{\rm ess sup }} |v|,\\
|v|_{m,p,A}&=\bigg(\sum_{|\alpha|=m}\|D^{\alpha}v\|_{0,p,A}^p\bigg)^{1\slash p}\qquad &&1\le p<\infty,\\
\|v\|_{m,p,A}&=\bigg(\sum_{j=0}^m|v|_{m,p,A}^p\bigg)^{1\slash p}.
\end{alignat*}

If $A$ is the whole domain, i.e., $A=\Omega$, then $\|\cdot\|_{H^k}, \|\cdot\|_{L^k}$ are used to simplify the notations $\|\cdot\|_{H^k(\Omega)}, \|\cdot\|_{L^k(\Omega)}$ respectively. Besides, assume $\T_h$ to be a family of quasi-uniform triangulations of domain $\Omega$, and $\mathcal{E}_h$ to be a collection of edges, then for any triangle $K\in\T_h$, define the following mesh dependent semi-norm, norm and inner product
\begin{align*}
|v|_{j,p,h}&=\bigg(\sum_{K\in\T_h}|v|_{j,p,K}^p\bigg)^{1\slash p},\\
\|v\|_{j,p,h}&=\bigg(\sum_{K\in\T_h}\|v\|_{j,p,K}^p\bigg)^{1\slash p},\\
(w,v)_h&=\sum_{K\in\T_h}\int_Kw(x)v(x)dx.
\end{align*}

Theoretically, the $tanh$ profile of the initial condition $u_0$ is required to prove the relations between the Cahn-Hilliard equation and the Hele-Shaw flow \cite{Alikakos94,Chen94}. Because of the $tanh$ profile, the following assumptions can be made on the initial condition, and they were used to derive a priori estimates for the solution of problem \eqref{eq20170504_1}--\eqref{eq20170504_5} \cite{feng2016analysis,Feng_Prohl04,Feng_Prohl05,li2015numerical}.

{\bf General Assumption} (GA)
\begin{itemize}
\item[(1)] Assume that $m_0\in (-1,1)$ where
\begin{equation}\label{eq2.2}
m_0:=\frac{1}{|\Omega|}\int_{\Omega}u_0(x)dx. 
\end{equation}
\item[(2)] There exists a nonnegative constant $\sigma_1$ such that
\begin{equation}\label{eq2.3}
J_{\epsilon}(u_0)\leq C\epsilon^{-2\sigma_1}.
\end{equation}
\item[(3)]
There exists nonnegative constants $\sigma_2$, $\sigma_3$ and $\sigma_4$ such that
\begin{align}\label{eq2.4}
\big\|-\epsilon\Delta u_0 +\epsilon^{-1} f(u_0)\big\|_{H^{\ell}} \leq C\epsilon^{-\sigma_{2+\ell}}\qquad
\ell=0,1,2.
\end{align}
\end{itemize}

Under the above assumptions, the following a priori estimates of the solution were proved in \cite{feng2016analysis,Feng_Prohl04,Feng_Prohl05,li2015numerical}.

\begin{theorem}\label{prop2.1}
The solution $u$ of problem \eqref{eq20170504_1}--\eqref{eq20170504_5} satisfies the following energy estimate:
\begin{align}
&\esssupT  \Bigl( \frac{\epsilon}{2}\|\nabla u\|_{L^2}^2 +\frac{1}{\epsilon}\|F(u)\|_{L^1} \Bigr)
+\begin{cases} \int_{0}^{T}\|u_t(s)\|_{H^{-1}}^2\, ds\\
\int_{0}^{T}\|\nabla w(s)\|_{L^2}^2\,ds 
\end{cases}  
\leq J_{\epsilon}(u_0)\label{eq2.5}.
\end{align}
Moreover, suppose that \eqref{eq2.2}--\eqref{eq2.4} hold, $u_0\in H^4(\Omega)$ and $\p\Omega\in C^{2,1}$, 
then $u$ satisfies the additional estimates:
\begin{align}
&\frac{1}{|\Omega|}\int_{\Omega}u(x,t)\, dx=m_0 \quad\forall t\geq 0, \label{eq2.8}\\
%
%
&\esssupT\|\nabla\Delta u\|_{L^2}\leq C\epsilon^{-\max\{\sigma_1+\frac{5}{2},\sigma_3+1\}}.\label{eq2.13_add}
%
\end{align}
Furthermore, if there exists $\sigma_5>0$ such that
\begin{equation}\label{eq2.17}
\mathop{\rm{lim}}_{s\rightarrow0^{+}}\limits\|\nabla u_t(s)\|_{L^2}\leq C\epsilon^{-\sigma_5},
\end{equation}
then there holds
\begin{align}
%
&\int_0^{T}\|u_{tt}\|_{H^{-1}}^2ds \leq C\tilde\rho_1(\epsilon),\label{eq2.19}
%
\end{align}
where
\begin{align*}
\tilde\rho_1(\epsilon) &:=\epsilon^{-\frac{1}{2}\max\{2\sigma_1+5,2\sigma_3+2\}
-\max\{2\sigma_1+\frac{13}{2},2\sigma_3+\frac{7}{2},2\sigma_2+4\}+1} +\epsilon^{-2\sigma_5+1}\\
&\qquad 
+\epsilon^{-\max\{2\sigma_1+7,2\sigma_3+4\}+1}.
\end{align*}
\end{theorem}

The next lemma gives an $\epsilon$-independent low bound for the principal eigenvalue of the linearized Cahn-Hilliard operator, and a proof of this lemma can be found in \cite{Chen94}. 
\begin{lemma}\label{lem3.4}
Suppose that \eqref{eq2.2}--\eqref{eq2.4} hold. Given a smooth initial curve/surface $\Gamma_0$, 
let $u_0$ be a smooth function satisfying $\Gamma_0 = \{x\in\Omega; u_0(x)=0\}$ and some profile 
described in \cite{Chen94}. Let $u$ be the solution to problem \eqref{eq20170504_1}--\eqref{eq20170504_5}. 
Define $\mathcal{L}_{CH}$ as
\begin{equation}\label{eq3.25a}
\mathcal{L}_{CH} := \Delta\left(\eps\Delta-\frac{1}{\eps}f'(u)I\right).
\end{equation}
Then there exists $0<\epsilon_0<<1$ and a positive constant $C_0$ such that the principle 
eigenvalue of the linearized Cahn-Hilliard operator $\mathcal{L}_{CH}$ satisfies 
\begin{equation}\label{eq3.25}
\lambda_{CH}:=\mathop{\inf}_{\substack{0\neq\psi\in H^1{(\Omega)}\\ \Delta w=\psi}}
\limits\frac{\epsilon\|\nabla\psi\|_{L^2}^2+\frac{1}{\epsilon}(f'(u)\psi,\psi)}{\|\nabla w\|_{L^2}^2}\geq -C_0
\end{equation}
for $t\in [0,T]$ and $\eps\in (0,\eps_0)$.\\
\end{lemma}

\begin{remark}
\begin{enumerate}
\item A discrete version of the spectrum estimate of \eqref{eq3.25} on conforming finite element spaces was proved in \cite{Feng_Prohl04,Feng_Prohl05}, and a discrete version on discontinuous Galerkin finite element space was proved in \cite{feng2016analysis}. They play crucial roles in the proofs of the convergence of the numerical interfaces to the Hele-Shaw flow \cite{Feng_Prohl04,Feng_Prohl05,feng2016analysis}.


\item In the assumption, the initial function $u_0$ should be chosen to satisfy some profile to guarantee the convergence results. A simple function satisfying this profile is $u_0=\tanh(\frac{d_0(x)}{\eps})$, where $d_0(x)$ denotes the signed distance function to the initial interface $\Gamma_0$. Assume $u$ is an arbitrary function, instead of being the solution of the Cahn-Hilliard equation, we can find a low bound of $\lambda_{CH}$, which depends on $\frac{1}{\epsilon}$ polynomially, by interpolating $L^2(\Omega)$ space to $H^1(\Omega)$ and $H^{-1}(\Omega)$ spaces.
\end{enumerate}
\end{remark}

The classical discrete Gronwall's inequality is a main technique to derive the error estimates of fully discretized scheme for partial differential equation (PDE) problems. However, for many nonlinear PDE problems, the classical discrete Gronwall's inequality can not be applied because of nonlinearity. Instead, a generalized version discrete Gronwall's inequality is needed. In case of the power (or Bernoulli-type) nonlinearity, a generalized continuous Gronwall's inequality was proved in \cite{feng2005posteriori}, and its discrete counterpart is stated below. The proof of this generalized discrete Gronwall's inequality can be found in \cite{Pachpatte}.

\begin{lemma}\label{lem2.1}
Let $\{S_{\ell} \}_{\ell\geq 1}$ be a positive nondecreasing sequence and 
$\{b_{\ell}\}_{\ell\geq 1}$ and $\{k_{\ell}\}_{\ell\geq 1}$ be nonnegative sequences, 
and $p>1$ be a constant. If
\begin{eqnarray}\label{eq2.13}
&S_{\ell+1}-S_{\ell}\leq b_{\ell}S_{\ell}+k_{\ell}S^p_{\ell} \qquad\mbox{for \ } \ell\geq 1,
\\ \label{eq2.14}
&S^{1-p}_{1}+(1-p)\mathop{\sum}\limits_{s=1}^{\ell-1}k_{s}a^{1-p}_{s+1}>0
\qquad\mbox{for \ } \ell\geq 2,
\end{eqnarray}
then
\begin{equation}\label{eq2.15}
S_{\ell}\leq \frac{1}{a_{\ell}} \Bigg\{S^{1-p}_{1}+(1-p)
\sum_{s=1}^{\ell-1}k_{s}a^{1-p}_{s+1}\Bigg\}^{\frac{1}{1-p}}\qquad\text{for \ }\ell\geq 2,  
\end{equation}
where
\begin{equation}\label{eq2.16}
a_{\ell} := \prod_{s=1}^{\ell-1} \frac{1}{1+b_{s}} \qquad\mbox{for \ } \ell\geq 2.
\end{equation}
\end{lemma}

Denote $L^2_0(\Ome)$ as the space of functions with zero mean, then for $\Phi\in L^2(\Ome)$, let $u := -\Delta^{-1}\Phi \in H^2(\Ome)\cap L^2_0(\Ome)$ such that
\begin{alignat}{2}
-\Delta u &= \Phi&&\qquad \mathrm{in}\ \Omega,\notag\\
\frac{\partial u}{\partial n}&= 0&&\qquad \mathrm{on}\ \partial\Omega.\notag
\end{alignat}

Then we have
\begin{align}\label{eq6}
-(\nabla\Delta^{-1}\Phi,\nabla v) = (\Phi,v)\quad \mathrm{in}\ \Omega\qquad\forall v\in H^1(\Ome)\cap L^2_0(\Ome).
\end{align}

For $v\in L_0^2(\Ome)$ and $\Phi\in L_0^2(\Ome)$, define the continuous $H^{-1}$ inner product by
\begin{align}\label{eq7}
(\Phi, v)_{H^{-1}} := (\nabla\Delta^{-1}\Phi,\nabla\Delta^{-1}v) = (\Phi,-\Delta^{-1}v) = (v,-\Delta^{-1}\Phi).
\end{align}

When $\Phi\in L^2_0(\Ome)$, define the induced continuous $H^{-1}$ norm by
\begin{align}
\|\Phi\|_{H^{-1}} := \sqrt{(\Phi, \Phi)_{H^{-1}}} =\|\nabla\Delta^{-1}\Phi\|_{L^2}.
\end{align}

%
%
%

Next define the Morley element spaces $S_h$ below \cite{brenner1999convergence,brenner2013morley,elliott1989nonconforming}:
\begin{center}
$S^h=\{v_h\in L^{\infty}(\D): v_h\in P_2(K), v_h$ is continuous at the vertices of all triangles, and $\frac{\partial v_h}{\partial n}$ is continuous at the midpoints of interelement edges of triangles\}.
\end{center}

Through the the paper, we assume
\begin{align}\label{eq20180425_5}
\|u_h^n\|_{L^{\infty}}\le C\epsilon^{-\gamma_1},
\end{align}
where $u_h^n$ is defined in \eqref{eq20170504_11}--\eqref{eq20170504_12} and $\gamma_1$ is a constant. Theoretically there is no analysis to prove the discrete maximum principle for the Cahn-Hilliard equation. However, numerically we can verify \eqref{eq20180425_5} for many initial conditions. In Section 4, two examples are given, and we find $\gamma_1=0$ and $C=1$ in these cases.


We use the following notation
\begin{center}
$H^j_E(\Omega)=\{v\in H^j(\Omega): \frac{\partial v}{\partial n}=0$ on $\partial\Omega$\}\qquad j=1, 2, 3.
\end{center}

Corresponding to $H^j_E(\Omega)$, define $S^h_E$ as the subspace of $S^h$ below: 
\begin{center}
$S^h_E=\{v_h\in S^h: \frac{\partial v_h}{\partial n}=0$ at the midpoints of the edges on $\partial\Omega$\}.
\end{center}

To the end, the enriching operator $\E$ is restated \cite{brenner1996two,brenner1999convergence,brenner2013morley}. Let $\widetilde{S}_E^h$ be the Hsieh-Clough-Tocher macro element space, which is an enriched space of the Morley finite element space $S_E^h$. Let $p$ and $m$ be the internal vertices and midpoints of triangles $\T_h$. Define $\E: S_E^h\rightarrow \widetilde{S}_E^h$ by
\begin{align*}
(\E v)(p) &= v(p),\\
\frac{\p\E v}{\p n}(m) &= \frac{\p v}{\p n}(m),\\
(\p^{\beta}(\E v))(p) &= \text{average of } (\p^{\beta}v_i)(p)\qquad |\beta|=1,
\end{align*}
where $v_i=v|_{T_i}$ and triangle $T_i$ contains $p$ as a vertex.

Define the interpolation operator $I_h: H^2_E(\Omega)\rightarrow S_E^h$ such that
\begin{align*}
(I_h v)(p)&=v(p),\\
\frac{\p I_h v}{\p n}(m)&=\frac{1}{|e|}\int_e\frac{\p v}{\p n}dS,
\end{align*}
where $p$ ranges over the internal vertices of all the triangles $T$, and $m$ ranges over the midpoints of all the edges $e$.

It can be proved that \cite{brenner1996two,brenner1999convergence,brenner2013morley,elliott1989nonconforming}
\begin{alignat}{2}\label{eq20170812_6}
|v-I_hv|_{j,p,K}&\le Ch^{3-j}|v|_{3,p,K}\qquad&&\forall K\in\mathcal{T}_h,\quad\forall v\in H^3(K),\quad j=0,1,2,\\
\|\E v-v\|_{j,2,h}&\le Ch^{2-j}|v|_{2,2,h}\quad&&\forall v\in S_E^h,\quad j=0,1,2.\label{eq20171006_1}
\end{alignat}

\section{Fully Discrete Approximation}
In this section, the Morley element is used to discretize the fourth order Cahn-Hilliard problem \eqref{eq20170504_1}--\eqref{eq20170504_5}. Different kinds of discrete inverse Laplace operators are defined in order to derive the discrete energy law and error estimates. The optimal $\|\cdot\|_{2,2,h}$ error orders are obtained under a weaker regularity assumption, i.e., $v\in H^{3,h}(\Omega)$. This can be considered as a generalization of the regularity assumption in paper \cite{elliott1989nonconforming}. Besides, it is proved that the error bounds depend on $\epsilon^{-1}$ in lower order polynomial, instead of in exponential order. The crux part to prove the error bounds is to prove the generalized coercivity result in the Morley element space, where the enriched finite element space is used as a bridge.

\subsection{Formulation}\label{subsec3_1}
The weak form of \eqref{eq20170504_1}--\eqref{eq20170504_3} is to seek $u(\cdot,t)\in H^2_E(\D)$ such that
\begin{align}\label{eq20170504_7}
(u_t,v)+\epsilon a(u,v)&=\frac{1}{\epsilon}(\nabla f(u),\nabla v)\quad\forall v\in H_E^2(\D),\\
u(\cdot,0)&=u_0\in H_E^2(\D),
\end{align}
where the bilinear form $a(\cdot,\cdot)$ is defined as
\begin{align}\label{eq20170504_8}
a(u,v)=\int_{\D}\Delta u\Delta v+\bigl(\frac{\partial^2u}{\partial x\partial y}\frac{\partial^2v}{\partial x\partial y}-\frac12\frac{\partial^2u}{\partial x^2}\frac{\partial^2v}{\partial y^2}-\frac12\frac{\partial^2u}{\partial y^2}\frac{\partial^2v}{\partial x^2}\bigr)dxdy
\end{align}
with Poisson's ratio set to $\frac12$.

It can be verified that \cite{lascaux1975some} for any $w\in H^2(\Omega)$,
\begin{align*}
a(w,w)=\frac12(\|\Delta w\|_{0,2,\Omega}^2+|w|_{2,2,\Omega}^2),
\end{align*}

and when $w, z$ are sufficiently smooth,
\begin{align*}
a(w,z)&=\int_{\Omega}\Delta^2w\,z\,dxdy-\int_{\p\Omega}\frac{\p\Delta w}{\p n}\,zdS\\
&\qquad+\int_{\p\Omega}\bigg(\Delta w-\frac12\frac{\p^2w}{\p s^2}\bigg)\frac{\p z}{\p n}dS+\frac12\int_{\p\Omega}\frac{\p^2w}{\p n\p s}\frac{\p z}{\p s}dS,
\end{align*}
where $n, s$ denote the normal and tangential directions respectively.

Define the following spaces
\begin{alignat*}{2}
H^{3,h}(\Omega)&=S^h\oplus  H^3(\Omega),\qquad H_E^{3,h}(\Omega)&&=S_E^h\oplus  H_E^3(\Omega),\\
H^{2,h}(\Omega)&=S^h\oplus  H^2(\Omega),\qquad H_E^{2,h}(\Omega)&&=S_E^h\oplus  H_E^2(\Omega),\\
H^{1,h}(\Omega)&=S^h\oplus  H^1(\Omega),\qquad H_E^{1,h}(\Omega)&&=S_E^h\oplus  H_E^1(\Omega),
\end{alignat*}
where, for instance,
\begin{align*}
S_E^h\oplus  H_E^2(\Omega)=\{u+v: u\in S_E^h\ \ \text{and}\ \ v\in H_E^2(\Omega)\}.
\end{align*}

Next define the discrete bilinear form
\begin{align}\label{eq20170504_9}
a_h(u,v)&=\sum_{K\in\mathcal{T}_h}\int_K\Delta u\Delta v+\bigl(\frac{\partial^2u}{\partial x\partial y}\frac{\partial^2v}{\partial x\partial y}-\frac12\frac{\partial^2u}{\partial x^2}\frac{\partial^2v}{\partial y^2}-\frac12\frac{\partial^2u}{\partial y^2}\frac{\partial^2v}{\partial x^2}\bigr)dxdy.
\end{align}

To introduce the elliptic projection $P_h$ \cite{elliott1989nonconforming}, we first define
\begin{align*}
R=\bigl\{v\in H_E^2(\D): \Delta v\in H_E^2(\D)\bigr\}.
\end{align*}

Then for arbitrary $v\in R$, define the following elliptic projection $P_h$ by seeking $P_hv\in S_E^h$ such that
\begin{align}\label{eq20170504_14}
\tilde b_h(P_hv,w)=(\epsilon\Delta^2v-\frac{1}{\epsilon}\div(f'(u)\nabla v)+\alpha v,w)\qquad\forall w\in S_E^h,
\end{align}
where
\begin{align}\label{eq20170504_15}
\tilde b_h(v,w)=\epsilon a_h(v,w)+\frac{1}{\epsilon}(f'(u)\nabla v,\nabla w)_h+\alpha(v,w).
\end{align}
Notice here $\alpha>\frac{C}{\epsilon^3}$ should be chosen to guarantee the coercivity of $\tilde b_h(v,w)$ because by the proof of Lemma 2.4 in \cite{elliott1989nonconforming}, when $z\in H^{2,h}(\Omega)$, we have
\begin{align*}
\tilde b_h(z,z)=&\frac{\epsilon}{2}(\|\Delta z\|_{0,2,h}^2+|z|_{2,2,h}^2)+\frac{1}{\epsilon}(f'(u)\nabla z,\nabla z)_h+\alpha(z,z)\\
\ge&\frac{\epsilon}{2}(\|\Delta z\|_{0,2,h}^2+|z|_{2,2,h}^2)-\frac{1}{\epsilon}(\nabla z,\nabla z)_h+\alpha(z,z)\\
\ge&\frac{\epsilon}{2}(\|\Delta z\|_{0,2,h}^2+|z|_{2,2,h}^2)-\frac{1}{\epsilon}(\nabla z,\nabla z)_h+[C(\alpha\epsilon)^{\frac12}(\nabla z,\nabla z)_h-\frac{\epsilon}{4}|z|_{2,2,h}^2].
\end{align*}

Based on the above bilinear form, our fully discrete Galerkin method is to find $u_h^n\in S^h_E$ such that 
\begin{align}\label{eq20170504_11}
(d_tu_h^{n},v_h)+\epsilon a_h(u_h^{n},v_h)+\frac{1}{\epsilon}(\nabla f(u_h^{n}),\nabla v_h)_h&=0\quad\forall v_h\in S^h_E,\\
u_h^0&=u_0^h\in S^h_E,\label{eq20170504_12}
\end{align}
where the difference operator $d_tu_h^{n} := \frac{u_h^{n}-u_h^{n-1}}{k}$, and $u_0^h := P_hu(t_0)$ .

\subsection{The $\|\cdot\|_{2,2,h}$ and $\|\cdot\|_{1,2,h}$ errors under weaker regularity assumptions}\label{subsec3_3}
In section 5 of paper \cite{elliott1989nonconforming}, the projection errors in $\|\cdot\|_{2,2,h}$ and $\|\cdot\|_{1,2,h}$ norms are proved under the assumption that the exact solution $u\in H^4(\Omega)$. In this paper, $\hat\Delta_h^{-1}\zeta$ is defined in \eqref{eq20170725_1}, which can be considered as a novel projection of $\Delta^{-1}\zeta$ where $\zeta\in S_E^h$, and we also give the error bounds between $\Delta^{-1}\zeta$ and $\hat\Delta_h^{-1}\zeta$ under the assumption that $\Delta^{-1}\zeta\in H^2(\Omega)\cap H^{3,h}(\Omega)$. In this case, notice here $\Delta^{-1}\zeta$ does not need to be related to the exact solution, even we define the bilinear form to be equal to the right-hand side (see Remark \ref{rem20171001_1} below for details). 

First we cite Lemma 2.5 in \cite{elliott1989nonconforming}, which will be used in this paper.
\begin{lemma}\label{lem20170723_1}
Let $z\in H_E^{2,h}(\Omega)$ and $w\in H^2_E(\Omega)\cap W^{3,p}(\Omega)$, and define $B_h(w,z)$ by
\begin{align*}
B_h(w,z) = \sum_{K\in\mathcal{T}_h}\int_{\partial K}\bigg(\Delta w\frac{\partial z}{\partial n}+\frac12\frac{\partial^2 w}{\partial n\partial s}\frac{\partial z}{\partial s}-\frac12\frac{\partial^2 w}{\partial s^2}\frac{\partial z}{\partial n}\bigg)dS,
\end{align*}
then we have
\begin{align*}
|B_h(w,z)|\le Ch|w|_{3,2,h}|z|_{2,2,h}.
\end{align*}
\end{lemma}

Next some mesh-dependent discrete inverse
Laplace operators are given here. Define space $W_h$ by
\begin{align*}
W_h=\{w_h\in L^2(\Omega)|w_h\ \text{is a piecewise polynomial with degree $\le 6$ on each triangle K}\}.
\end{align*}

Then we can define the discrete inverse 
Laplace operator $\underline{\Delta}_h^{-1}: L^2(\Omega)\rightarrow {W}_h$ as follows: given 
$\zeta\in L^2(\Omega)$, define $\underline{\Delta}_h^{-1}\zeta\in{W}_h$ such that
\begin{equation}\label{eq3.2}
(\nabla\underline{\Delta}_h^{-1}\zeta,\nabla w_h)_h+(\underline{\Delta}_h^{-1}\zeta, w_h)=(\nabla\Delta^{-1}\zeta,\nabla w_h)_h+(\Delta^{-1}\zeta,w_h) \qquad \forall\, w_h\in {W}_h.
\end{equation}

Therefore, $-\underline{\Delta}_h^{-1}\zeta$ can be considered as a projection of $-\Delta^{-1}\zeta$.

As a comparison, we define the discrete inverse 
Laplace operator $\Delta_h^{-1}: L^2(\Omega)\rightarrow {W}_h$ as follows: given $\zeta\in L^2(\Omega)$, define $\Delta_h^{-1}\zeta\in{W}_h$ such that
\begin{equation}\label{eq20171106_8}
(\nabla\Delta_h^{-1}\zeta,\nabla w_h)_h+(\Delta_h^{-1}\zeta, w_h)=-(\zeta, w_h)+(\Delta^{-1}\zeta,w_h) \qquad \forall\, w_h\in {W}_h.
\end{equation}

Furthermore, define $\widetilde\Delta_h^{-1}, \hat\Delta_h^{-1}: S_E^h\rightarrow S_E^h$ as follows: given 
$\zeta\in S_E^h$, let $\widetilde\Delta_h^{-1}\zeta, \hat\Delta_h^{-1}\zeta\in S_E^h$ such that
\begin{alignat}{2}\label{eq3.2b}
b_h(-\widetilde\Delta_h^{-1}\zeta,w_h)&=(\nabla\zeta,\nabla w_h)_h+\beta(-\Delta^{-1}\zeta,w_h)\qquad \forall\, w_h\in S_E^h,\\
b_h(-\hat\Delta_h^{-1}\zeta,w_h)&=(\nabla\zeta,\nabla w_h)_h+B_h(-\Delta^{-1}\zeta,w_h)\label{eq20170725_1}\\
&\qquad+\beta(-\Delta^{-1}\zeta,w_h) \qquad \forall\, w_h\in S_E^h,\notag
\end{alignat}
where $b_h(u,v):=a_h(u,v)+\beta(u,v)$, and $\beta$ is a positive number to guarantee the coercivity of $b_h(u,v)$, i.e., $\beta=1$ by the proof of Lemma 2.4 in \cite{elliott1989nonconforming}.


For any $v\in H^{3}(\Omega)$, it always holds that
\begin{alignat}{2}
b_h(v,\eta)&=-(\nabla\Delta v,\nabla \eta)_h+B_h(v,\eta)+\beta(v,\eta)\label{eq20170710_2}\\
&:=F_h(\eta)\qquad\forall \eta\in H^{2,h}_E(\Omega).\notag
\end{alignat}

Corresponding to operator $\hat\Delta_h^{-1}$, for any $v\in H^{3}(\Omega)$, define $v_h\in S_E^h$ by
\begin{alignat}{2}
b_h(v_h,\xi)&=-(\nabla\Delta v,\nabla \xi)_h+B_h(v,\xi)+\beta(v,\xi)\label{eq20170710_3}\\
&:={F}_h(\xi)\qquad\forall \xi\in S_E^h.\notag
\end{alignat}

Corresponding to operator $\widetilde\Delta_h^{-1}$, for any $v\in H^{3}(\Omega)$, define $v_h\in S_E^h$ by
\begin{alignat}{2}
b_h(v_h,\xi)&=-(\nabla\Delta v,\nabla \xi)_h+\beta(v,\xi)\label{eq20180604_9}\\
&:=\hat{F}_h(\xi)\qquad\forall \xi\in S_E^h.\notag
\end{alignat}

By equations \eqref{eq20170710_2} and \eqref{eq20170725_1}, we know 
\begin{alignat}{2}\label{eq20170723_1}
b_h(-\Delta^{-1}\zeta,\eta)&=(\nabla \zeta,\nabla \eta)_h-B_h(\Delta^{-1}\zeta,\eta)\\
&-\beta(\Delta^{-1}\zeta,\eta)\qquad\forall \eta\in H^{2,h}_E(\Omega),\notag\\
b_h(-\hat\Delta_h^{-1}\zeta,\xi)&=(\nabla \zeta,\nabla \xi)_h-B_h(\Delta^{-1}\zeta,\xi)\label{eq20171001_6}\\
&-\beta(\Delta^{-1}\zeta,\eta)\qquad\forall \xi\in S_E^h.\notag
\end{alignat}

Then it is ready to prove the optimal error estimates of $\|\hat\Delta_h^{-1}u-\Delta^{-1}u\|_{1,2,h}$ and $\|\hat\Delta_h^{-1}u-\Delta^{-1}u\|_{2,2,h}$ when $u\in S_E^h$. Notice $u\in L^2(\Omega)$, but $u$ may not be in $ H^1(\Omega)$. Instead of using properties of the Morley elements (Lemmas 2.1--2.6 in \cite{elliott1989nonconforming}), the enriching operator is perfectly employed to derive the upper bounds.

\begin{lemma}\label{lem20170604_1_add}
Assume $\hat\Delta_h^{-1}$ is defined in \eqref{eq20170725_1} and $u\in S_E^h$, then
\begin{align*}
\|\hat\Delta_h^{-1}u-\Delta^{-1}u\|_{2,2,h}\le Ch\|u\|_{1,2,h}.
\end{align*}
\end{lemma}
%
%
%
%

\begin{proof}
Using \eqref{eq20170710_2} and \eqref{eq20170710_3}, we obtain 
\begin{align}\label{eq20170710_8_add}
&\quad b_h(v-v_h,v-v_h)\\
&=b_h(v-v_h,v-I_hv)+b_h(v,I_hv-v_h)-F_h(I_hv-v_h)\notag\\
&= b_h(v-v_h,v-I_hv)\notag\\
&\le \|v-v_h\|_{2,2,h}\|v-I_hv\|_{2,2,h}.\notag
\end{align}

Let $v=\Delta^{-1}\E u$ and $v_h=\hat\Delta_h^{-1}u$, by \eqref{eq20170812_6} and the elliptic regularity theory, we have
\begin{align}\label{eq20180604_1}
\|\Delta^{-1}\E u-\hat\Delta_h^{-1}u\|_{2,2,h}&\le\|\Delta^{-1}\E u-I_h\Delta^{-1}\E u\|_{2,2,h}\\
&\le Ch|\Delta^{-1}\E u|_{H^3}\notag\\
&\le Ch\|\E u\|_{H^1}\notag\\
&\le Ch\|u\|_{1,2,h}\notag,
\end{align}
where the last inequality uses \eqref{eq20171006_1}, the inverse inequality and the triangle inequality.

On the other hand, using the elliptic regularity theory and \eqref{eq20171006_1}, we have
\begin{align}\label{eq20180604_2}
\|\Delta^{-1}\E u-\Delta^{-1}u\|_{H^2}&\le\|u- \E u\|_{L^2}\\
&\le Ch|u|_{1,2,h}\notag.
\end{align}

Combining \eqref{eq20180604_1} and \eqref{eq20180604_2}, and using the triangle inequality, the theorem can be obtained immediately.
\end{proof}

The following lemma is a direct result of Lemma \ref{lem20170604_1_add}.
\begin{lemma}\label{lem20170710_1_add}
Assume $\hat\Delta_h^{-1}$ is defined in \eqref{eq20170725_1} and $u\in S_E^h$, then
\begin{align*}
\|\hat\Delta_h^{-1}u-\Delta^{-1}u\|_{1,2,h}\le Ch\|u\|_{1,2,h}.
\end{align*}
\end{lemma}

\begin{remark}\label{rem20180604}
\begin{enumerate}
\item In \eqref{eq20170710_8_add}, if the enriching operator is not introduced, i.e., let $v=\Delta^{-1}u$ and $v_h=\hat\Delta_h^{-1}u$, we can only obtain
\begin{align*}
\|\Delta^{-1}u-\hat\Delta_h^{-1}u\|_{2,2,h}\le Ch|\Delta^{-1}u|_{3,2,h}.
\end{align*}
In the following part of this paper, $u$ in in the Morley element space $S_E^h$, which is not in $H^1(\Omega)$, so the inequality below may be very hard to obtain
\begin{align*}
|\Delta^{-1}u|_{3,2,h}\le Ch\|u\|_{1,2,h}.
\end{align*}
\end{enumerate}
\end{remark}

Next we prove the error between $\Delta^{-1}\zeta$ and $\widetilde\Delta_h^{-1}\zeta$. 

\begin{lemma}\label{lemma_3.7_add}
Assume $\widetilde\Delta_h^{-1}$ is defined in \eqref{eq3.2b} and $\zeta\in S_E^h$, then
\begin{align*}
\|\Delta^{-1}\zeta-\widetilde\Delta_h^{-1}\zeta\|_{2,2,h}&\le Ch\|\zeta\|_{1,2,h}.
\end{align*}
\end{lemma}
\begin{proof}
By \eqref{eq20180604_9} and \eqref{eq20170710_8_add}, we have
\begin{align}\label{eq20180604_7}
\|v-v_h\|_{2,2,h}^2&=Cb_h(v-v_h,v-I_hv)+C[b_h(v,I_hv-v_h)-\hat{F}_h(I_hv-v_h)]\\
&\le Cb_h(v-v_h,v-I_hv)+CB_h(v,I_hv-v_h)\notag.
\end{align}

Notice we will set $v_h=-\widetilde\Delta_h^{-1}\zeta$, so equation \eqref{eq20180604_9} is used, then $\hat{F}_h$, instead of $F_h$, appears in equation \eqref{eq20180604_7}.

Using Lemma \ref{lem20170723_1}, we have
\begin{align}\label{eq20180604_10}
&\quad\|v-v_h\|_{2,2,h}^2\\
&\le C\|v-v_h\|_{2,2,h}\|v-I_hv\|_{2,2,h}+Ch|\Delta^{-1}\E\zeta|_{3,2,h}\|I_hv-v_h\|_{2,2,h}\notag\\
&\le Ch|\zeta|_{1,2,h}\|v-v_h\|_{2,2,h}+Ch\|\zeta\|_{1,2,h}(\|I_hv-v\|_{2,2,h}+\|v-v_h\|_{2,2,h})\notag.
\end{align}

Let $v=-\Delta^{-1}\E\zeta, v_h=-\widetilde\Delta_h^{-1}\zeta$, then
\begin{align}\label{eq20180604_11}
&\quad\|\Delta^{-1}\E\zeta-\widetilde\Delta_h^{-1}\zeta\|_{2,2,h}\le Ch\|\zeta\|_{1,2,h}.
\end{align}

Combining \eqref{eq20180604_2} and \eqref{eq20180604_11}, we get the conclusion.
\end{proof}

The following bound is a direct result from Lemma \ref{lemma_3.7_add}. 
\begin{lemma}\label{lem20170914_1}
Assume $\widetilde\Delta_h^{-1}$ is defined in \eqref{eq3.2b} and $\zeta\in S_E^h$, then
\begin{align*}
\|\Delta^{-1}\zeta-\widetilde\Delta_h^{-1}\zeta\|_{1,2,h}&\le Ch\|\zeta\|_{1,2,h}.
\end{align*}
\end{lemma}

\begin{remark}
\begin{enumerate}
\item Using the theory of enriching operators, instead of the properties of the Morley elements, we can applied the theory in this paper to other nonconforming elements, based on their enriched conforming elements. If properties of the Morley elements are used to get the bound of $\|\Delta^{-1}\zeta-\hat\Delta_h^{-1}\zeta\|_{2,2,h}$, we can obtain
\begin{align*}
\|\Delta^{-1}\zeta-\hat\Delta_h^{-1}\zeta\|_{2,2,h}&\le Ch|\Delta^{-1}\zeta|_{3,2,h}+Ch|\Delta^{-1}\zeta|_{4,2,h}.
\end{align*}
\item Suppose we use the properties of the Morley elements, and we do not employ the enriching operators. When $\zeta\in S_E^h$, $\Delta^{-1}\zeta$ may not be in $H^2_E(\Omega)\cap W^{3,p}(\Omega)$ so that Lemma \ref{lem20170723_1} can not be applied. Then we can only prove the following lemma when the Poisson's ratio is $1$, which is not physical.
\begin{lemma}\label{lem20180425_2}
Let $z\in H_E^{2,h}(\Omega)$ and $\Delta w\in S_E^h$, and when Poisson's ratio is $1$, we have
\begin{align*}
|B_h(w,z)|\le Ch(\|\Delta w\|_{1,2,h}\|z\|_{2,2,h}+\|\Delta w\|_{2,2,h}\|z\|_{1,2,h}).
\end{align*}
\end{lemma}
\begin{proof}
By Lemma 2.3 in \cite{elliott1989nonconforming} and the inverse inequality, we get when $w, z\in H_E^{2,h}(\Omega)$, and at least one of them is in $S_E^h$, then
\begin{align*}
|\sum_{K\in\mathcal{T}_h}\int_{\partial K}\frac{\partial z}{\partial n}w|\le Ch(\|w\|_{1,2,h}\|z\|_{2,2,h}+\|w\|_{2,2,h}\|z\|_{1,2,h}).
\end{align*}

When Poisson's ratio is $1$, bilinear form $a_h(u,v)$ in \eqref{eq20170504_9} and $B_h(w,z)$ become
\begin{align}\label{eq20180426_6}
a_h(u,v)&=\sum_{K\in\mathcal{T}_h}\int_K\Delta u\Delta v,\\
B_h(w,z)& = \sum_{K\in\mathcal{T}_h}\int_{\partial K}\Delta w\frac{\partial z}{\partial n}dS.\label{eq20180426_7}
\end{align}

Then we have
\begin{align*}
|B_h(w,z)|\le Ch(\|\Delta w\|_{1,2,h}\|z\|_{2,2,h}+\|\Delta w\|_{2,2,h}\|z\|_{1,2,h}).
\end{align*}
\end{proof}
\end{enumerate}
\end{remark}

Before we give the relations between operators $\underline{\Delta}_h^{-1}$, $\Delta_h^{-1}$ and $\Delta^{-1}$, we need an extra lemma.
\begin{lemma}\label{lem20180424_1}
The operators $\underline{\Delta}_h^{-1}$ and $\Delta_h^{-1}$ are defined in \eqref{eq3.2} and \eqref{eq20171106_8}, then for any $\zeta\in L^2(\Omega)$, we have
\begin{align*}
&(\nabla\Delta^{-1}\zeta,\nabla(\underline{\Delta}_h^{-1}\zeta-\Delta_h^{-1}\zeta))_h+(\zeta, \underline{\Delta}_h^{-1}\zeta-\Delta_h^{-1}\zeta)\\
&\quad\le Ch\|\underline{\Delta}_h^{-1}\zeta-\Delta_h^{-1}\zeta\|_{1,2,h}\|\zeta\|_{0,2,h}.
\end{align*}
\end{lemma}
\begin{proof}
Define an elliptic projection $P_1: L^2(\Omega)\rightarrow V\cap L^2_0$ by
\begin{align*}
(\nabla\zeta-\nabla P_1\zeta,\nabla v)_h=0\qquad\forall v\in V\cap L^2_0,
\end{align*}
where $V$ can be a conforming space consisting of piecewise polynomials.

Define another elliptic projection $P_2: H^2(\Omega)\rightarrow V\cap L^2_0$ by
\begin{align*}
(\nabla\Delta^{-1}\zeta-\nabla P_2(\Delta^{-1}\zeta),\nabla v)=0\qquad\forall v\in V\cap L^2_0.
\end{align*}

Then we have
\begin{align}\label{eq20180425_1}
&\quad(\nabla\Delta^{-1}\zeta,\nabla(\underline{\Delta}_h^{-1}\zeta-\Delta_h^{-1}\zeta))_h+(\zeta, \underline{\Delta}_h^{-1}\zeta-\Delta_h^{-1}\zeta)\\
&\le(\nabla\Delta^{-1}\zeta-\nabla P_2(\Delta^{-1}\zeta),\nabla(\underline{\Delta}_h^{-1}\zeta-\Delta_h^{-1}\zeta))_h+(\zeta, \underline{\Delta}_h^{-1}\zeta-\Delta_h^{-1}\zeta)\notag\\
&\quad+(\nabla P_2(\Delta^{-1}\zeta),\nabla(\underline{\Delta}_h^{-1}\zeta-\Delta_h^{-1}\zeta))_h\notag\\
&\le(\nabla\Delta^{-1}\zeta-\nabla P_2(\Delta^{-1}\zeta),\nabla(\underline{\Delta}_h^{-1}\zeta-\Delta_h^{-1}\zeta))_h+(\zeta, \underline{\Delta}_h^{-1}\zeta-\Delta_h^{-1}\zeta)\notag\\
&\quad+(\nabla P_2(\Delta^{-1}\zeta),\nabla P_1(\underline{\Delta}_h^{-1}\zeta-\Delta_h^{-1}\zeta))\notag\\
&\le(\nabla\Delta^{-1}\zeta-\nabla P_2(\Delta^{-1}\zeta),\nabla(\underline{\Delta}_h^{-1}\zeta-\Delta_h^{-1}\zeta))_h+(\zeta, \underline{\Delta}_h^{-1}\zeta-\Delta_h^{-1}\zeta)\notag\\
&\quad+(\nabla P_2(\Delta^{-1}\zeta)-\nabla\Delta^{-1}\zeta,\nabla P_1(\underline{\Delta}_h^{-1}\zeta-\Delta_h^{-1}\zeta))+(\nabla\Delta^{-1}\zeta,\nabla P_1(\underline{\Delta}_h^{-1}\zeta-\Delta_h^{-1}\zeta))\notag\\
&\le(\nabla\Delta^{-1}\zeta-\nabla P_2(\Delta^{-1}\zeta),\nabla(\underline{\Delta}_h^{-1}\zeta-\Delta_h^{-1}\zeta))_h+(\nabla P_2(\Delta^{-1}\zeta)-\nabla\Delta^{-1}\zeta,\notag\\
&\quad\nabla P_1(\underline{\Delta}_h^{-1}\zeta-\Delta_h^{-1}\zeta))+(\zeta, \underline{\Delta}_h^{-1}\zeta-\Delta_h^{-1}\zeta-P_1(\underline{\Delta}_h^{-1}\zeta-\Delta_h^{-1}\zeta))\notag\\
&\le Ch|\underline{\Delta}_h^{-1}\zeta-\Delta_h^{-1}\zeta|_{1,2,h}\|\zeta\|_{0,2,h}+Ch^2\|\underline{\Delta}_h^{-1}\zeta-\Delta_h^{-1}\zeta\|_{2,2,h}\|\zeta\|_{0,2,h}\notag\\
&\le Ch\|\underline{\Delta}_h^{-1}\zeta-\Delta_h^{-1}\zeta\|_{1,2,h}\|\zeta\|_{0,2,h}.\notag
\end{align}
\end{proof}

Next some lemmas related to operators $\underline{\Delta}_h^{-1}$ and $\Delta_h^{-1}$ are proved below.

\begin{lemma}\label{lem20171106_1}
Assume $\underline\Delta_h^{-1}$ is defined in \eqref{eq3.2} and $\zeta\in L^2(\Omega)$, then
\begin{align}
\|\underline\Delta_h^{-1}\zeta-\Delta_h^{-1}\zeta\|_{1,2,h}\le Ch\|\zeta\|_{0,2,h}.\notag
\end{align}
\end{lemma}
\begin{proof}
Subtracting \eqref{eq20171106_8} from \eqref{eq3.2}, choosing $w_h=\underline{\Delta}_h^{-1}\zeta-\Delta_h^{-1}\zeta$, and using Lemma \ref{lem20180424_1}, we obtain
\begin{align}\label{eq20171106_9}
&\|\nabla\underline{\Delta}_h^{-1}\zeta-\nabla\Delta_h^{-1}\zeta\|_{0,2,h}^2+\|\underline{\Delta}_h^{-1}\zeta-\Delta_h^{-1}\zeta\|_{0,2,h}^2\\
&\qquad=(\nabla\Delta^{-1}\zeta,\nabla(\underline{\Delta}_h^{-1}\zeta-\Delta_h^{-1}\zeta))_h+(\zeta, \underline{\Delta}_h^{-1}\zeta-\Delta_h^{-1}\zeta)\notag\\
&\qquad\le Ch\|\underline{\Delta}_h^{-1}\zeta-\Delta_h^{-1}\zeta\|_{1,2,h}\|\zeta\|_{0,2,h}\notag.
\end{align}

Then the lemma is proved.
\end{proof}

\begin{lemma}\label{lem20171106_2}
Assume $\Delta_h^{-1}$ is defined in \eqref{eq20171106_8} and $\zeta\in L^2(\Omega)$, then
\begin{align}
\|\Delta_h^{-1}\zeta-\Delta^{-1}\zeta\|_{1,2,h}\le Ch\|\zeta\|_{0,2,h}.\notag
\end{align}
\end{lemma}
\begin{proof}
Observe $\underline\Delta_h^{-1}$ is a projection of $\Delta^{-1}$, then we can prove
\begin{align}\label{eq20171106_11}
\|\underline\Delta_h^{-1}\zeta-\Delta^{-1}\zeta\|_{1,2,h}\le Ch^j\|\zeta\|_{j-1,2,h},\quad j=1,2.
\end{align}
Combining \eqref{eq20171106_11} and Lemma \ref{lem20171106_1}, and using the triangle inequality, this lemma can be proved.
\end{proof}

\begin{remark}\label{rem20171001_1}
\begin{enumerate}


\item In Lemmas \ref{lem20170604_1_add}--\ref{lem20170914_1}, the regularity requirement on $u$ is $\Delta^{-1}u\in H^{2}(\Omega)\cap H^{3,h}(\Omega)$. It is proved that the error bounds can depend on norm $\|\cdot\|_{3,2,h}$, instead of norm $\|\cdot\|_{4,2,\Omega}$. Hence the lemmas in this subsection can be considered as a generalization of the error bounds in \cite{elliott1989nonconforming}.

\item The idea of proposing the bilinear form $b_h(\cdot,\cdot)$ defined in equation \eqref{eq20170710_2} is that \eqref{eq20170710_2} automatically holds for $\eta\in H^{2,h}_E(\Omega)$, but the bilinear form (5.2b) in \cite{elliott1989nonconforming} holds under the condition that $u\in H^{4}(\Omega)$. This is the main reason why the regularity requirement in paper \cite{elliott1989nonconforming} can be removed. Another advantage of using this generalized projection is the proofs of the error estimates can be simplified (see proofs of Lemmas \ref{lem20170604_1_add}--\ref{lem20170710_1_add}).

\end{enumerate}
\end{remark}

%
%

\subsection{The discrete energy law and the discrete stability results}

In order to mimic the continuous energy law in \eqref{eq2.5}, we consider the discrete energy law under some mesh constraints in this subsection. A lemma is needed to prove the discrete energy law.
First we give the bound of the $L^2$ norm interpolation.
\begin{lemma}
For any $\zeta\in L^2(\Omega)$, then
\begin{align*}
\|\zeta\|_{L^2}^2\le\|\nabla\Delta_h^{-1}\zeta\|_{L^2}^2+\|\nabla \zeta\|_{0,2,h}^2.
\end{align*}
\end{lemma}
\begin{proof}
Testing \eqref{eq20171106_8} by $\zeta$, and using Lemma \ref{lem20171106_2}, we obtain
\begin{align}\label{eq20180421_1}
\|\zeta\|_{L^2}^2&=(-\nabla\Delta_h^{-1}\zeta,\nabla \zeta)_h+(\Delta^{-1}\zeta-\Delta_h^{-1}\zeta,\zeta)\\
&\le\frac12\|\nabla\Delta_h^{-1}\zeta\|_{L^2}^2+\frac12\|\nabla \zeta\|_{0,2,h}^2+Ch\|\zeta\|_{L^2}\|\zeta\|_{L^2}\notag.
\end{align}
When $Ch\le\frac12$, the lemma is proved.
\end{proof}
\begin{remark}\label{rem20180421_1}
Combine \eqref{eq20180421_1} and Lemma \ref{lem20170914_1}, we can easily prove
\begin{align*}
\|\zeta\|_{L^2}^2&\le\|\nabla\Delta^{-1}\zeta\|_{L^2}^2+\|\nabla \zeta\|_{0,2,h}^2,\\
\|\zeta\|_{L^2}^2&\le \frac{1}{a}\|\nabla\Delta^{-1}\zeta\|_{L^2}^2+Ca\|\nabla \zeta\|_{0,2,h}^2,\notag\\
\|\zeta\|_{L^2}^2&\le \frac{1}{a}\|\nabla\widetilde\Delta^{-1}\zeta\|_{L^2}^2+Ca\|\nabla \zeta\|_{0,2,h}^2.\notag
\end{align*}
\end{remark}

The discrete energy law is proved below.
\begin{theorem}\label{thm20171006_2}
Under the assumption \eqref{eq20180425_5} and the following mesh constraints
\begin{align*}
k&\ge C\frac{h^2}{\epsilon},\\
k&\ge C\frac{h^2}{\epsilon^{4\gamma_1+3}},\\
k&\ge C\epsilon\beta^2h^2,
\end{align*}
the following energy holds
\begin{align*}
&J^h_{\epsilon}(u_h^n)+\frac{k}{8}\sum_{n=1}^{\ell}\|\nabla \Delta^{-1}d_tu_h^n\|_{0,2,h}^2\\
&\qquad+\frac{\epsilon k^2}{8}\sum_{n=1}^{\ell}\|\nabla d_tu_h^n\|_{0,2,h}^2+\frac{k^2}{4\eps}\sum_{n=1}^{\ell}\|d_t(|u_h^{n}|^2-1)\|_{0,2,h}^2\le CJ^h_{\epsilon}(u_h^0),
\end{align*}
where
\begin{align*}
J^h_{\epsilon}(v) = \frac{\epsilon}{2}\|\nabla v\|_{0,2,h}^2+\frac{1}{4\eps} \| v^2-1\|_{0,2,h}^2.
\end{align*}
\end{theorem}

\begin{proof}
Taking $v_h=-\widetilde\Delta_h^{-1}(u_h^{n}-u_h^{n-1})$ as the test function in \eqref{eq20170504_11}, then we have
\begin{align}\label{eq20170729_11}
&(d_tu_h^{n},-\widetilde\Delta_h^{-1}(u_h^{n}-u_h^{n-1}))+\epsilon a_h(u_h^{n},-\widetilde\Delta_h^{-1}(u_h^{n}-u_h^{n-1}))\\
&\qquad+\frac{1}{\epsilon}(\nabla f(u_h^{n}),-\nabla\widetilde\Delta_h^{-1}(u_h^{n}-u_h^{n-1}))_h=0.\notag
\end{align}

The first term on the left-hand side of \eqref{eq20170729_11} can be written as
\begin{align}\label{eq20170730_1}
M_1&=k(\nabla \Delta^{-1}d_tu_h^n,\nabla \Delta^{-1}d_tu_h^n)_h+k(\nabla \Delta_h^{-1}d_tu_h^n-\nabla \Delta^{-1}d_tu_h^n,\nabla \Delta^{-1}d_tu_h^n)_h\\
&\quad+k(\nabla \Delta_h^{-1}d_tu_h^n,\nabla \widetilde\Delta_h^{-1}d_tu_h^n-\nabla \Delta^{-1}d_tu_h^n)_h\notag\\
&\quad+k(\widetilde\Delta_h^{-1}d_tu_h^n,\Delta_h^{-1}d_tu_h^n- \Delta^{-1}d_tu_h^n)\notag\\
&\ge k\|\nabla \Delta^{-1}d_tu_h^n\|_{0,2,h}^2-\bigl[\frac{k}{8}\|\nabla \Delta^{-1}d_tu_h^n\|_{0,2,h}^2+Ckh^2\|d_tu_h^n\|_{0,2,h}\bigr]\notag\\
&\quad-\bigl[\frac{k}{8}\|\nabla\Delta^{-1}d_tu_h^n\|_{0,2,h}^2+Ckh^2\|d_tu_h^n\|_{0,2,h}^2+Ckh^2\|d_tu_h^n\|_{1,2,h}^2\bigl]\notag\\
&\quad-\bigl[\frac{k}{8}\|\nabla\Delta^{-1}d_tu_h^n\|_{0,2,h}^2+Ckh^2\|d_tu_h^n\|_{0,2,h}^2+Ckh^2\|d_tu_h^n\|_{1,2,h}^2\bigl]\notag\\
&\ge\frac{k}{2}\|\nabla \Delta^{-1}d_tu_h^n\|_{0,2,h}^2-Ckh^2\|\nabla d_tu_h^n\|_{0,2,h}^2\notag,
\end{align}
where Remark \ref{rem20180421_1} is used in the last inequality.

The second term on the left-hand side of \eqref{eq20170729_11} can be written as
\begin{align}\label{eq20170730_2}
M_2 &= \epsilon(\nabla u_h^n,\nabla (u_h^n-u_h^{n-1}))_h+\epsilon\beta k(u_h^n,\Delta^{-1}d_tu_h^n-\widetilde\Delta_h^{-1}d_tu_h^n)\\
&\ge\frac{\epsilon}{2}\|\nabla u_h^n\|_{0,2,h}^2-\frac{\epsilon}{2}\|\nabla u_h^{n-1}\|_{0,2,h}^2+\frac{\epsilon k^2}{2}\|\nabla d_tu_h^n\|_{0,2,h}^2\notag\\
&\qquad-C\epsilon k\|\nabla u_h^n\|_{0,2,h}^2-C\epsilon\beta^2kh^2\|d_tu_h^n\|_{1,2,h}^2\notag\\
&\ge\frac{\epsilon}{2}\|\nabla u_h^n\|_{0,2,h}^2-\frac{\epsilon}{2}\|\nabla u_h^{n-1}\|_{0,2,h}^2-C\epsilon k\|\nabla u_h^n\|_{0,2,h}^2\notag\\
&\qquad+\frac{3\epsilon k^2}{8}\|\nabla d_tu_h^n\|_{0,2,h}^2-\frac{k}{8}\|\nabla \Delta^{-1}d_tu_h^n\|_{0,2,h}^2\notag,
\end{align}
where the last inequality hold under the restriction $k\ge C\beta^2h^2$.


We now bound the third term on the left-hand side of \eqref{eq20170729_11} from below.  
We consider the case $f^{n}=(u_h^{n})^3 - u_h^{n}$, and it can be written as 
\begin{align*}
f^{n} 
&=u_h^{n}\bigl(|u_h^{n}|^2-1\bigr)\\
&=\frac12 \bigl( (u_h^{n}+u_h^{n-1})+kd_t u_h^{n} \bigr)
\bigl(|u_h^{n}|^2-1\bigr).
\end{align*}
A direct calculation then yields \cite{feng2014analysis}
\begin{align}\label{eq3.18}
\frac{1}{\eps}\bigl(f^{n},d_t u_h^{n}\bigr)_{h}
&\geq \frac{1}{4\eps} d_t\| |u_h^{n}|^2-1\|_{0,2,h}^2 \\
&\quad
+\frac{k}{4\eps}\|d_t(|u_h^{n}|^2-1)\|_{0,2,h}^2
-\frac{k}{2\eps}\|d_t u_h^{n}\|_{0,2,h}^2. \nonumber
\end{align}



The third term on the left-hand side of \eqref{eq20170729_11} can be written as
\begin{align}\label{aaa_eq20170730_3}
M_3&=\frac{k}{\epsilon}(\nabla f(u_h^{n}),-\nabla\Delta_h^{-1}d_tu_h^{n})_h+\frac{k}{\epsilon}(\nabla f(u_h^{n}),\nabla\Delta_h^{-1}d_tu_h^{n}-\nabla\Delta^{-1}d_tu_h^{n})_h\\
&\qquad+\frac{k}{\epsilon}(\nabla f(u_h^{n}),\nabla\Delta^{-1}d_tu_h^{n}-\nabla\widetilde\Delta_h^{-1}d_tu_h^{n})_h\notag\\
&\ge \frac{k}{\epsilon}(f(u_h^{n}),d_tu_h^{n})_h+\frac{k}{\epsilon}(\Delta^{-1}d_tu_h^{n}-\Delta_h^{-1}d_tu_h^{n},f(u_h^{n}))\notag\\
&\qquad- C\epsilon^{4\gamma_1}k\|\nabla f(u_h^n)\|_{0,2,h}^2-\frac{C}{\epsilon^{4\gamma_1+2}}kh^2\|d_tu_h^{n}\|_{0,2,h}^2\notag\\
&\qquad-C\epsilon^{4\gamma_1}k\|\nabla f(u_h^{n})\|_{0,2,h}^2-\frac{C}{\epsilon^{4\gamma_1+2}}kh^2\|d_tu_h^{n}\|_{1,2,h}^2\notag\\
&\ge \frac{k}{\epsilon}(f(u_h^{n}),d_tu_h^{n})_h-C\epsilon^{2\gamma_1-1}k\|f(u_h^{n})\|_{0,2,h}^2-\frac{C}{\epsilon^{2\gamma_1+1}}kh^2\|d_tu_h^{n}\|_{0,2,h}^2\notag\\
&\qquad-\frac{C}{\epsilon^{4\gamma_1+2}}kh^2\|d_tu_h^{n}\|_{0,2,h}^2-C\epsilon^{4\gamma_1}k\|\nabla f(u_h^{n})\|_{0,2,h}^2\notag\\
&\qquad-\frac{C}{\epsilon^{4\gamma_1+2}}kh^2\|\nabla d_tu_h^{n}\|_{0,2,h}^2-\frac{k}{16}\|\nabla \Delta^{-1}d_tu_h^n\|_{0,2,h}^2\notag\\
&\ge\bigl[\frac{k}{4\eps} d_t\| |u_h^{n}|^2-1\|_{0,2,h}^2+\frac{k^2}{4\eps}\|d_t(|u_h^{n}|^2-1)\|_{0,2,h}^2\bigl]-C\frac{k}{\epsilon}\|(u_h^n)^2-1\|_{0,2,h}^2\nonumber\\
&\quad-\bigl[\frac{k}{8}\|\nabla \Delta^{-1}d_tu_h^n\|_{0,2,h}^2+\frac{Ckh^2}{\epsilon^{4\gamma_1+2}}\|\nabla d_tu_h^{n}\|_{0,2,h}^2\bigl]-C\epsilon k\|\nabla u_h^n\|_{0,2,h}^2.\notag
\end{align}

Taking the summation over $n$ from $1$ to $\ell$, and restricting $k$ by letting $k\ge C\frac{h^2}{\epsilon^{4\gamma_1+3}}$, then the energy law can be obtained by the Gronwall's inequality.
\end{proof}

\begin{remark}\label{rem20171103_1}
\begin{enumerate}
\item The idea of proving this discrete energy law is to control bad terms in $M_1$ by terms $M_2$, which is different from the conforming Galerkin case \cite{Feng_Prohl04,Feng_Prohl05} and the discontinuous Galerkin case \cite{feng2016analysis}. This is one reason why there are some restrictions in this theorem.


\item The constant $C$ in the energy law can be chosen to approach 1 by restricting $k$ as the polynomial of $\epsilon$ more stringently. 


\end{enumerate}
\end{remark}

A lemma about summation by parts below is needed in this section.
\begin{lemma}\label{lem20180515_1}
Suppose $\{a_n\}_{n=0}^\ell$ and $\{b_n\}_{n=0}^\ell$ are two sequences, then
\begin{equation*}
\sum_{n=1}^\ell(a^n-a^{n-1},b^n)=(a^\ell,b^\ell)-(a^0,b^0)-\sum_{n=1}^\ell(a^{n-1},b^n-b^{n-1}).
\end{equation*}
\end{lemma}
\begin{proof} 
The lemma can be easily obtained by using the equality below
\begin{equation*}
\sum_{n=1}^\ell(a^{n-1},b^n-b^{n-1}) = \sum_{n=1}^\ell(a^{n-1},b^n)-\sum_{n=1}^\ell(a^{n},b^{n})+(a^\ell,b^\ell)-(a^0,b^0).
\end{equation*}
\end{proof}

Next we prove the $\|u_h^n\|_{2,2,h}$ stability results for the cases when $L^2$ in time (Theorem \ref{thm20180517_1}) and $L^{\infty}$ in time (Theorem \ref{thm20171006_3}) are considered, which will be used in proving the generalized coercivity result in the Morley element space.
\begin{theorem}\label{thm20180517_1}
Under the mesh constraints in Theorem \ref{thm20171006_2}, the following stability result holds
\begin{align*}
\frac12\|u_h^{\ell}\|_{0,2,h}^2+\frac{k}{2}\|d_tu_h^{n}\|_{0,2,h}^2+\epsilon k\sum_{n=1}^{\ell}\|u_h^{n}\|_{2,2,h}^2+\frac{3k}{\epsilon}\sum_{n=1}^{\ell}\|u_h^{n}\nabla u_h^{n}\|_{0,2,h}^2\le C\epsilon^{-2\sigma_1-2}\notag,
\end{align*}
where $C$ is also the $\epsilon$-independent constant.
\end{theorem}
\begin{proof}
Taking $v_h=u_h^{n}$ as the test function in \eqref{eq20170504_11}, then
\begin{align}\label{eq20180517_3}
(d_tu_h^{n},u_h^{n})+\epsilon a_h(u_h^{n},u_h^{n})+\frac{1}{\epsilon}(\nabla f(u_h^{n}),\nabla u_h^{n})_h=0.
\end{align}

The first term on the left-hand side of \eqref{eq20180517_3} can be written as
\begin{align}\label{eq20180517_4}
(d_tu_h^{n},u_h^{n})=\frac{1}{2k}\|u_h^{n}\|_{0,2,h}^2-\frac{1}{2k}\|u_h^{n-1}\|_{0,2,h}^2+\frac{1}{2k}\|u_h^{n}-u_h^{n-1}\|_{0,2,h}^2.
\end{align}

The third term on the left-hand side of \eqref{eq20180517_3} can be written as
\begin{align}\label{eq20180517_5}
\frac{1}{\epsilon}(\nabla f(u_h^{n}),\nabla u_h^{n})_h&=\frac{1}{\epsilon}((3(u_h^{n})^2-1)\nabla u_h^{n},\nabla u_h^{n})_h\\
&=\frac{3}{\epsilon}\|u_h^{n}\nabla u_h^{n}\|_{0,2,h}^2-\frac{1}{\epsilon}\|\nabla u_h^{n}\|_{0,2,h}^2.\notag
\end{align}

Taking the summation over $n$ from $1$ to $\ell$ on both sides of \eqref{eq20180517_3}, multiplying with $k$, and using Theorem \ref{thm20171006_2}, we obtain the conclusion.
\end{proof}

\begin{theorem}\label{thm20171006_3}
Under the mesh constraints in Theorem \ref{thm20171006_2}, and when $k\ge C\frac{h^4}{\epsilon^{4+4\gamma_1+2\sigma_1}}(\ln\,\frac1h)^2$ and $k\ge Ch^2$, the following stability result holds
\begin{align*}
\|u_h^{\ell}\|_{2,2,h}^2+\sum_{n=1}^\ell\|u_h^{n}-u_h^{n-1}\|_{2,2,h}^2+\sum_{n=1}^{\ell}\frac{\|u_h^{n}-u_h^{n-1}\|_{0,2,h}^2}{\epsilon k}
\le C\epsilon^{-2\gamma_2},
\end{align*}
where $\gamma_2:=2\gamma_1+\sigma_1+6$ and $C$ is the $\epsilon$-independent constant.
\end{theorem}
\begin{proof}
Taking $v_h=u_h^{n}-u_h^{n-1}$ as the test function in \eqref{eq20170504_11}, then
\begin{align}\label{eq20170814_1}
(d_tu_h^{n},u_h^{n}-u_h^{n-1})+\epsilon a_h(u_h^{n},u_h^{n}-u_h^{n-1})+\frac{1}{\epsilon}(\nabla f(u_h^{n}),\nabla (u_h^{n}-u_h^{n-1}))_h=0.
\end{align}

The first term on the left-hand side of \eqref{eq20170814_1} can be written as
\begin{align}\label{eq20170814_2}
(d_tu_h^{n},u_h^{n}-u_h^{n-1})=\frac1k\|u_h^{n}-u_h^{n-1}\|_{L^2}^2.
\end{align}

The second term on the left-hand side of \eqref{eq20170814_1} can be written as
\begin{align}\label{eq20170814_3}
\epsilon a_h(u_h^{n},u_h^{n}-u_h^{n-1})&=\frac{\epsilon}{2} a_h(u_h^{n},u_h^{n})-\frac{\epsilon}{2}a_h(u_h^{n-1},u_h^{n-1})\\
&\qquad+\frac{\epsilon}{2} a_h(u_h^{n}-u_h^{n-1},u_h^{n}-u_h^{n-1}).\notag
\end{align}

Using summation by parts in Lemma \ref{lem20180515_1} and integration by parts, then the summation of the third term on the left-hand side of \eqref{eq20170814_1} can be written as
\begin{align*}
&-\frac{1}{\epsilon}\sum_{n=1}^\ell(\nabla f(u_h^{n}),\nabla (u_h^{n}-u_h^{n-1}))_h\\
=& -\frac{1}{\epsilon}\sum_{n=1}^\ell\sum_{E\in\mathcal{E}_h}( f(u_h^{n}),\frac{\partial(u_h^{n}-u_h^{n-1})}{\partial n})_E+\frac{1}{\epsilon}\sum_{n=1}^\ell(f(u_h^{n}),\Delta (u_h^{n}-u_h^{n-1}))_h\notag\\
=& -\frac{1}{\epsilon}\sum_{n=1}^\ell\sum_{E\in\mathcal{E}_h}([\![f(u_h^{n})]\!],\bigl\{\frac{\partial(u_h^{n}-u_h^{n-1})}{\partial n}\bigr\})_E-\frac{1}{\epsilon}\sum_{n=1}^\ell\sum_{E\in\mathcal{E}_h}(\{f(u_h^{n})\},\bigl[\!\bigl[\frac{\partial(u_h^{n}-u_h^{n-1})}{\partial n}\bigr]\!\bigr])_E\notag\\
&+\frac{1}{\epsilon}(f(u_h^\ell),\Delta u_h^\ell)_h-\frac{1}{\epsilon}(f(u_h^0),\Delta u_h^0)_h-\frac{1}{\epsilon}\sum_{n=1}^\ell(f(u_h^{n})-f(u_h^{n-1}),\Delta u_h^{n-1})_h\notag\\
:=& T_1+T_2+T_3+T_4+T_5,\notag
\end{align*}
where $[\![\cdot]\!]$ and $\{\cdot\}$ denote the jump and the average along the mesh boundaries.

Using the inverse inequality and Theorem \ref{thm20180517_1}, when $k\ge C\frac{h^4}{\epsilon^{4+4\gamma_1+2\sigma_1}}(\ln\,\frac1h)^2$, we have
\begin{align}\label{eq20180517_7}
T_1&\le\frac{1}{\epsilon}\sum_{n=1}^\ell Ch^2|f(u_h^{n})|_{2,2,h}|u_h^{n}-u_h^{n-1}|_{1,\infty,h}\\
&\le\frac{\epsilon}{8}\sum_{n=1}^\ell|u_h^{n}-u_h^{n-1}|_{2,2,h}^2+C\frac{h^4}{\epsilon^{3}}\ln\,\frac1h\sum_{n=1}^\ell|(3(u_h^{n})^2-1)\Delta u_h^{n}+6u_h^{n}(\nabla u_h^{n})^2|_{0,2,h}^2\notag\\
&\le\frac{\epsilon}{8}\sum_{n=1}^\ell|u_h^{n}-u_h^{n-1}|_{2,2,h}^2+k\sum_{n=1}^\ell|\Delta u_h^{n}|_{0,2,h}^2+Ck\sum_{n=1}^\ell|u_h^{n}|_{2,2,h}^2\notag,
\end{align}
where the first inequality uses the proof of Lemma 2.6 in \cite{elliott1989nonconforming} before applying the inverse inequality.

When $k\ge Ch^2$, using Theorem \ref{thm20171006_2} and the idea of the proof of Lemma 2.1 in \cite{elliott1989nonconforming}, it holds for each element $K$, then the second term can be bounded by
\begin{align}\label{eq20180517_8}
T_2&\le Ch\sum_{n=1}^\ell|f(u_h^{n})|_{1,2,h}|u_h^{n}-u_h^{n-1}|_{2,2,h}\\
&\le \frac{\epsilon}{4} \sum_{n=1}^\ell a_h(u_h^{n}-u_h^{n-1},u_h^{n}-u_h^{n-1})+C\epsilon^{-4\gamma_1-1}k\sum_{n=1}^\ell|u_h^{n}|_{1,2,h}^2\notag\\
&\le \frac{\epsilon}{4} \sum_{n=1}^\ell a_h(u_h^{n}-u_h^{n-1},u_h^{n}-u_h^{n-1})+C\epsilon^{-4\gamma_1-2\sigma_1-2}\notag.
\end{align}


By Theorem \ref{thm20180517_1}, the third term and the fourth term can be bounded by
\begin{align}\label{eq20180517_9}
T_3+T_4\le \frac{\epsilon}{4}a_h(u_h^n,u_h^n)+C\epsilon^{-2\sigma_1-4\gamma_1-5}.
\end{align}

Using Theorem \ref{thm20180517_1}, the fifth term can be bounded by
\begin{align}\label{eq20180517_10}
T_5\le&\frac{1}{\epsilon}\sum_{n=1}^\ell(f(u_h^{n})-f(u_h^{n-1}),\Delta u_h^{n-1})_h\\
\le&\frac{1}{8k}\sum_{n=1}^\ell\|u_h^{n}-u_h^{n-1}\|_{L^2}^2+C\epsilon^{-4\gamma_1-2}k\sum_{n=1}^\ell|u_h^{n}|_{2,2,h}^2\notag\\
\le&\frac{1}{8k}\sum_{n=1}^\ell\|u_h^{n}-u_h^{n-1}\|_{L^2}^2+C\epsilon^{-4\gamma_1-2\sigma_1-5}\notag.
\end{align}

Taking the summation over $n$ from $1$ to $\ell$, and combining \eqref{eq20170814_2}--\eqref{eq20180517_10}, we have
\begin{align}
\epsilon\|u_h^{\ell}\|_{2,2,h}^2+\epsilon\sum_{n=1}^\ell\|u_h^{n}-u_h^{n-1}\|_{2,2,h}^2+\sum_{n=1}^{\ell}\frac{\|u_h^{n}-u_h^{n-1}\|_{0,2,h}^2}{k}
\le&C\epsilon^{-4\gamma_1-2\sigma_1-5}\notag.
\end{align}
\end{proof}

\begin{remark}
If $v_h=u_h^{n}$ or $v_h=u_h^{n}-u_h^{n-1}$ are chosen as the test function in \eqref{eq20170504_11}, we can only obtain the $L^2$ and $\|\cdot\|_{2,2,h}$ stability with upper bounds which are exponentially dependent on $\frac{1}{\epsilon}$.
\end{remark}


\subsection{The generalized coercivity result in the Morley element space}\label{subsec3_5}
Recall $\widetilde{S}_E^h$ is the Hsieh-Clough-Tocher macro element space. This $C^1$ conforming finite element space $\widetilde{S}_E^h$ is contained in $H^1(\Omega)$ space, so the following discrete spectrum estimate holds automatically.

\begin{lemma}\label{lem20170812_1}
Under the assumptions of Lemma \ref{lem3.4}, there exists an $\eps$-independent and $h$-independent constant $C_0>0$ such that for $\eps\in(0,1)$ and a.e. $t\in [0,T]$
\begin{align*}
\lambda_{CH}^{CONF}:=\mathop{\inf}_{\substack{0\neq\psi\in \widetilde{S}_E^h\\ \Delta w=\psi}}
\limits\frac{\epsilon\|\nabla\psi\|_{L^2}^2+\frac{1}{\epsilon}(f'(u(t))\psi,\psi)}{\|\nabla w\|_{L^2}^2}\geq -C_0
\end{align*}
for $t\in [0,T]$ and $\eps\in (0,\eps_0)$.\\
\end{lemma}

Before the generalized coercivity result is given, the following lemma is needed. It is about continuous $H^{-1}(\Omega)$ norm. 
\begin{lemma}\label{lem20171108_1}
The $H^{-1}$ norm has the following equivalent forms
\begin{align*}
\|\Phi\|_{H^{-1}}=\mathop{\sup}_{0\neq\xi\in H^1\cap L^2_0}\limits\frac{(\Phi,\xi)}{|\xi|_{H^1}}.
\end{align*}
\end{lemma}
\begin{proof}
By \eqref{eq6} and Holder's inequality,
\begin{align*}
(\Phi,\xi)=-(\nabla\Delta^{-1}\Phi,\nabla\xi)\le\|\nabla\Delta^{-1}\Phi\|_{L^2}\|\nabla\xi\|_{L^2},
\end{align*}
Then we have
\begin{align*}
\mathop{\sup}_{0\neq\xi\in H^1\cap L^2_0}\limits\frac{(\Phi,\xi)}{|\xi|_{H^1}}\le \|\Phi\|_{H^{-1}}.
\end{align*}
On the other hand, choose $\xi=-\Delta^{-1}\Phi$, then
\begin{align*}
\mathop{\sup}_{0\neq\xi\in H^1\cap L^2_0}\limits\frac{(\Phi,\xi)}{|\xi|_{H^1}}\ge\frac{(\nabla\Delta^{-1}\Phi,\nabla\Delta^{-1}\Phi)}{\|\nabla\Delta^{-1}\Phi\|_{L^2}}=\|\Phi\|_{H^{-1}}.
\end{align*}
Then the lemma is proved.
\end{proof}


Then we prove the generalized coercivity result in the Morley element space using the properties of the enriching operators.
\begin{theorem}\label{thm3.7_add}
Suppose there exist positive numbers $C_2>0$ and $\gamma_3>0$ such that the solution $u$ of 
problem \eqref{eq20170504_1}--\eqref{eq20170504_5} satisfies
\begin{equation}\label{eq3.23}
\|u-P_h u\|_{L^{\infty}((0,T);L^{\infty})}
\leq C_2 h \eps^{-\gamma_3},
\end{equation}
where the existence of $C_2$ and $\gamma_3$ can be guaranteed by imbedding the $L^{\infty}$ space to $H^2$ space.

Suppose $\psi\in S_E^h\cap L^2_0(\Omega)$ and the mesh constraints in Theorem \ref{thm20171006_3} hold, then there exists an $\eps$-independent and $h$-independent constant $C>0$ such that for $\eps\in(0,1)$ and a.e. $t\in [0,T]$
\begin{equation*}
N:=(\epsilon-\epsilon^4)(\nabla\psi,\nabla\psi)_h+
\frac{1}{\epsilon}(f'(P_hu(t))\psi,\psi)_h\geq -C\|\nabla\Delta^{-1}\psi\|_{L^2}^2-C\epsilon^{-2\gamma_2-4}h^2,
\end{equation*}
provided that $h$ satisfies the constraint
\begin{align}\label{eq3.24b}
h\leq (C_1C_2)^{-1}\eps^{\gamma_3+3},  
\end{align} 
where $C_1$ arises from the following equality:
\begin{align*}
C_1&:=\max_{|\xi|\le C_3}|f{''}(\xi)|.
\end{align*}
\end{theorem}

\begin{proof}
Based on the boundness of the exact solution of the Cahn-Hilliard equation, we assume there exists $C_3$ such that
\begin{align}\label{eq20170813_2add}
\|u\|_{L^{\infty}((0,T);L^{\infty})}\le C_3,\qquad \|P_hu\|_{L^{\infty}((0,T);L^{\infty})}\le C_3.
\end{align}

Then under the mesh constraint \eqref{eq3.24b}, we have
\begin{align*}
\|f'(P_hu(t))-f'(u(t))\|_{L^{\infty}((0,T);L^{\infty})}\le\epsilon^3.
\end{align*}

Then we have
\begin{align*}
\|f'(P_hu(t))\|_{L^{\infty}((0,T);L^{\infty})}\ge\|f'(u(t))\|_{L^{\infty}((0,T);L^{\infty})}-\epsilon^3.
\end{align*}

Then the term $N$ can be bounded by
\begin{align}\label{eq20170813_1}
N=&(\epsilon-\epsilon^4)(\nabla\psi,\nabla\psi)_h+
\frac{1}{\epsilon}(f'(P_hu(t)))\psi,\psi)_h\\
=&\epsilon^4(\nabla\psi,\nabla\psi)_h+2\epsilon^2(f'(P_hu(t))\psi,\psi)_h\notag\\
&\quad+(1-2\epsilon^3)\bigl[\epsilon(\nabla\psi,\nabla\psi)_h+\frac{1}{\epsilon}(f'(P_hu(t)))\psi,\psi)_h\bigr]\notag\\
\ge&\epsilon^4(\nabla\psi,\nabla\psi)_h+2\epsilon^2(f'(P_hu(t))\psi,\psi)_h-(1-2\epsilon^3)\epsilon^2(\psi,\psi)\notag\\
&\quad+(1-2\epsilon^3)\bigl[\epsilon(\nabla\psi,\nabla\psi)_h+\frac{1}{\epsilon}(f'(u(t))\psi,\psi)\bigr]\notag.
\end{align}

Besides, 
using the Lemma \ref{lem20171106_2} and Remark \ref{rem20180421_1}, we obtain
\begin{align}\label{eq20171108_1}
-C\epsilon^2(\psi,\psi)&=C\epsilon^2(\nabla\Delta_h^{-1}\psi,\nabla\psi)_{h}+C\epsilon^2(\Delta_h^{-1}\psi-\Delta^{-1}\psi,\psi)\\
&\ge-\frac{\epsilon^4}{8}(\nabla\psi,\nabla\psi)_h-C\|\nabla\Delta^{-1}\psi\|_{L^2}^2-Ch^2\|\psi\|_{L^2}^2\notag\\
&\quad-C\epsilon^2h^2\|\psi\|_{L^2}^2-\epsilon^2(\psi,\psi)_h\notag,\\
&\ge-\frac{\epsilon^4}{4}(\nabla\psi,\nabla\psi)_h-C\|\nabla\Delta^{-1}\psi\|_{L^2}^2\notag.
\end{align}

Then we have
\begin{align}\label{eq20170813_4}
N&\ge\frac{5\epsilon^4}{8}(\nabla\psi,\nabla\psi)_h-C\|\nabla\Delta^{-1}\psi\|_{L^2}^2\\
&\qquad+(1-2\epsilon^3)\bigl[\epsilon(\nabla\psi,\nabla\psi)_h+\frac{1}{\epsilon}(f'(u(t)))\psi,\psi)\bigr].\notag
\end{align}

If $\psi=P_hu(t_n)-u_h^n$, then by Theorem \ref{thm20171006_3}, we have
\begin{align*}
\|\psi\|_{2,2,h}\le C\epsilon^{-\gamma_2}.
\end{align*}

Define $\widetilde{\psi}$ by $\widetilde{\psi}=\E\psi$, then by \eqref{eq20171006_1}, we have
\begin{align}\label{eq20170821_4_add}
|\widetilde{\psi}-\psi|_{1,2,h}&\le C\epsilon^{-\gamma_2}h\\
|\widetilde{\psi}|_{1,2,h}&\le |\psi|_{1,2,h}+C\epsilon^{-\gamma_2}h.\label{eq20170821_4}
\end{align}

%


Using \eqref{eq20170821_4} and Remark \ref{rem20180421_1}, we obtain
\begin{align}\label{eq20170821_5}
(1-2\epsilon^3+\frac{\epsilon^3}{8})\epsilon(\nabla\psi,\nabla\psi)_{0,2,h}&\ge (1-2\epsilon^3)\epsilon(\nabla\widetilde\psi,\nabla\widetilde\psi)_{0,2,h}-C\epsilon^{-2\gamma_2-2}h^2\\
(1-2\epsilon^3)\frac{1}{\epsilon}(f'(u(t)))\psi,\psi)&=(1-2\epsilon^3)\frac{1}{\epsilon}(f'(u(t)))\widetilde\psi,\widetilde\psi)\label{eq20170821_6}\\
&\quad+(1-2\epsilon^3)\frac{1}{\epsilon}(f'(u(t))),\psi^2-\widetilde\psi^2)\notag\\
&\ge(1-2\epsilon^3)\frac{1}{\epsilon}(f'(u(t)))\widetilde\psi,\widetilde\psi)\notag\\
&\quad-\frac{C}{\epsilon^4}\|\psi-\widetilde\psi\|_{L^2}^2-\epsilon^2\|\psi\|_{L^2}^2\notag\\
&\ge(1-2\epsilon^3)\frac{1}{\epsilon}(f'(u(t)))\widetilde\psi,\widetilde\psi)\notag\\
&\quad-C\epsilon^{-2\gamma_2-4}h^2-\frac{\epsilon^4}{8}(\nabla\psi,\nabla\psi)_h\notag-C\|\nabla\Delta^{-1}\psi\|_{L^2}^2.
\end{align}

Using \eqref{eq20171006_1}, the definition of operator $\Delta^{-1}$ and Lemma \ref{lem20171108_1}, we obtain
\begin{align}\label{eq20171001_8}
\|\nabla\Delta^{-1}\widetilde\psi-\nabla\Delta^{-1}\psi\|_{L^2}&\le\|\widetilde\psi-\psi\|_{L^2}\\
&\le Ch^2\epsilon^{-\gamma_2}.\notag
\end{align}

When $h\le C\epsilon^{\gamma_1}$, using \eqref{eq20170821_5}--\eqref{eq20171001_8}, and Lemma \ref{lem20170812_1}, equation \eqref{eq20170813_4} can be bounded by
\begin{align}\label{eq20170813_7}
N\ge&\frac{\epsilon^4}{2}(\nabla\psi,\nabla\psi)_h-C\|\nabla\Delta^{-1}\psi\|_{L^2}^2+(1-2\epsilon^3)\bigl[\epsilon(\nabla\widetilde\psi,\nabla\widetilde\psi)_h\\
&+\frac{1}{\epsilon}(f'(u(t)))\widetilde\psi,\widetilde\psi)_h\bigr]-C\epsilon^{-2\gamma_2-4}h^2\notag\\
\ge&\frac{\epsilon^4}{2}(\nabla\psi,\nabla\psi)_h-C\|\nabla\Delta^{-1}\psi\|_{L^2}^2-C\|\nabla\Delta^{-1}\widetilde\psi\|_{L^2}^2-C\epsilon^{-2\gamma_2-4}h^2\notag\\
\ge&-C\|\nabla\Delta^{-1}\psi\|_{L^2}^2-C\epsilon^{-2\gamma_2-4}h^2\notag.
\end{align}
\end{proof}


\subsection{The error estimates in polynomial of $\frac{1}{\epsilon}$}
In this subsection, an error estimate of $\widetilde{\Delta}_h^{-1}(P_hu(t_n)-u_h^n)$ with polynomial dependence on $\frac{1}{\epsilon}$ is derived, and as corollaries, error estimates of $\hat\Delta_h^{-1}(P_hu(t_n)-u_h^n)$, $\Delta^{-1}(P_hu(t_n)-u_h^n)$, $\underline\Delta_h^{-1}(P_hu(t_n)-u_h^n)$ and $\Delta_h^{-1}(P_hu(t_n)-u_h^n)$ with polynomial dependence on $\frac{1}{\epsilon}$ are also given.
\begin{theorem}\label{thm20171007_1}
Suppose $u$ is the solution of \eqref{eq20170504_1}--\eqref{eq20170504_5}, $u_h^n$ is the numerical solution of scheme \eqref{eq20170504_11}--\eqref{eq20170504_12}, and assumption \eqref{eq20180425_5} holds. Define $\theta^n:=P_hu(t_n)-u_h^n$, then under following mesh constraints
\begin{alignat*}{3}
h&\le C\epsilon^2k, \qquad&&k\le C\epsilon^{3\sigma_1+13},\\
h&\le C\epsilon^{4\gamma_1+4},\qquad&&h \leq (C_1C_2)^{-1}\eps^{\gamma_3+3},
\end{alignat*}

we have the following error estimate
\begin{align*}
&\frac{1}{4}\|\nabla \widetilde\Delta_h^{-1}\theta^{\ell}\|_{0,2,h}^2+\frac{k^2}{4}\sum_{n=1}^\ell\|\nabla \widetilde\Delta_h^{-1}d_t\theta^n\|_{0,2,h}^2+\frac{\epsilon^4k}{16}\sum_{n=1}^\ell(\nabla\theta^n,\nabla\theta^n)_h\\
&\qquad+\frac{k}{\epsilon}\sum_{n=1}^\ell\|\theta^n\|_{0,4,h}^4\leq C(\tilde\rho_0(\epsilon)|\ln h|h^2+\tilde\rho_1(\epsilon)k^2).
\end{align*}
\end{theorem}

\begin{proof}

Using \eqref{eq20170504_11}--\eqref{eq20170504_12}, $\forall v_h\in S^h_E$, we have
\begin{align}\label{eq20170504_17}
&(d_t\theta^n,v_h)+\epsilon a_h(\theta^n,v_h)\\
=&[(d_tP_hu(t_n),v_h)+\epsilon a_h(P_hu(t_n),v_h)]-[(d_tu_h^n,v_h)+\epsilon a_h(u_h^n,v_h)]\notag\\
=&-(d_t\rho^n,v_h)+(d_tu(t_n)+\epsilon\Delta^2u(t_n)-\frac{1}{\epsilon}\Delta f(u(t_n))+\alpha u(t_n),v_h)\notag\\
&\quad-[\frac{1}{\epsilon}(f'(u(t_n))\nabla P_hu(t_n),\nabla v_h)_h+\alpha(P_hu(t_n),v_h)]+\frac{1}{\epsilon}(\nabla f(u_h^n),\nabla v_h)_h\notag\\
=&(-d_t\rho^n+\alpha\rho^n,v_h)-\frac{1}{\epsilon}(f'(u(t_n))\nabla P_hu(t_n)-\nabla f(u_h^n),\nabla v_h)_h\notag\\
&\quad+(R(u_{tt},n),v_h),\notag
\end{align}
where
\begin{equation*} 
R(u_{tt};n):=\frac{1}{k}\int^{t_n}_{t_{n-1}}(s-t_{n-1})u_{tt}(s)\,ds.
\end{equation*}
It follows from \eqref{eq2.19} that
\begin{align*}
k\sum_{n=1}^{\ell}\|R(u_{tt};n)\|_{H^{-1}}^2 
&\leq \frac{1}{k}\sum_{n=1}^{\ell} \Bigl(\int^{t_n}_{t_{n-1}}(s-t_{n-1})^2\,ds\Bigr)
\Bigl(\int^{t_n}_{t_{n-1}}\|u_{tt}(s)\|_{H^{-1}}^2\,ds\Bigr)\\
&\leq Ck^2\tilde\rho_1(\eps).
\end{align*}

Taking $v_h=-\widetilde\Delta_h^{-1}\theta^n$ in \eqref{eq20170504_17}, we have
\begin{align}\label{eq20170507_1}
&(d_t\theta^n,-\widetilde\Delta_h^{-1}\theta^n)+\epsilon a_h(\theta^n,-\widetilde\Delta_h^{-1}\theta^n)+(R(u_{tt},n),-\widetilde\Delta_h^{-1}\theta^n)\\
&=(-d_t\rho^n+\alpha\rho^n,-\widetilde\Delta_h^{-1}\theta^n)-\frac{1}{\epsilon}(f'(u(t_n))\nabla P_hu(t_n)-\nabla f(u_h^n),-\nabla \widetilde\Delta_h^{-1}\theta^n)_h.\notag
\end{align}

By the definition of $\Delta_h^{-1}$ and $\widetilde\Delta_h^{-1}$, then we have 
\begin{align}\label{eq20170507_2}
&(\nabla \Delta_h^{-1}d_t\theta^n,\nabla\widetilde\Delta_h^{-1}\theta^n)_h+\epsilon(\nabla\theta^n,\nabla\theta^n)_h+\beta\epsilon(\widetilde\Delta_h^{-1}\theta^n-\Delta^{-1}\theta^n,\theta^n)\\
&\qquad+\frac{1}{\epsilon}(\nabla f(P_hu(t_n))-\nabla f(u_h^n)),-\nabla \widetilde\Delta_h^{-1}\theta^n)_h+(R(u_{tt},n),-\widetilde\Delta_h^{-1}\theta^n)\notag\\
&=\frac{1}{\epsilon}(f'(P_hu(t_n))\nabla P_hu(t_n)-f'(u(t_n))\nabla P_hu(t_n),-\nabla \widetilde\Delta_h^{-1}\theta^n)\notag\\
&\qquad+(-d_t\rho^n+\alpha\rho^n,-\widetilde\Delta_h^{-1}\theta^n)+(\Delta^{-1}d_t\theta^n-\Delta_h^{-1}d_t\theta^n,\widetilde\Delta_h^{-1}\theta^n)_h.\notag
\end{align}

When $h\le C\epsilon^2k$, using Remark \ref{rem20180421_1}, then the first term on the left-hand side of \eqref{eq20170507_2} can be bounded by
 \begin{align}\label{eq20170714_1}
L_1&=(\nabla \widetilde\Delta^{-1}_hd_t\theta^n,\nabla \widetilde\Delta^{-1}_h\theta^n)_h+(\nabla \Delta_h^{-1}d_t\theta^n-\nabla \widetilde\Delta^{-1}_hd_t\theta^n,\nabla \Delta_h^{-1}\theta^n)_h\\
&\quad+(\nabla \Delta_h^{-1}d_t\theta^n-\nabla \widetilde\Delta^{-1}_hd_t\theta^n,\nabla \widetilde\Delta_h^{-1}\theta^n-\nabla \Delta_h^{-1}\theta^n)_h\notag\\
&\ge\bigl[\frac{k}{2}\|\nabla \widetilde\Delta^{-1}_hd_t\theta^n\|_{0,2,h}^2+\frac{1}{2k}\|\nabla\widetilde\Delta^{-1}_h\theta^n\|_{0,2,h}^2-\frac{1}{2k}\|\nabla\widetilde\Delta^{-1}_h\theta^{n-1}\|_{0,2,h}^2\bigr]\notag\\
&\quad-\bigl[Ch^2\|d_t\theta^n\|_{1,2,h}^2+\|\nabla\Delta^{-1}\theta^n\|_{0,2,h}^2+Ch^2\|\theta^n\|_{0,2,h}^2\bigl]\notag\\
&\quad-\bigl[Ch^4\|d_t\theta^n\|_{1,2,h}^2+\frac{\epsilon^4}{32}\|\nabla\theta^n\|_{0,2,h}^2+C\|\nabla\widetilde{\Delta}_h^{-1}\theta^n\|_{0,2,h}^2\bigl]\notag\\
&\ge\bigl[\frac{k}{2}\|\nabla \widetilde\Delta^{-1}_hd_t\theta^n\|_{0,2,h}^2+\frac{1}{2k}\|\nabla\widetilde\Delta^{-1}_h\theta^n\|_{0,2,h}^2-\frac{1}{2k}\|\nabla\widetilde\Delta^{-1}_h\theta^{n-1}\|_{0,2,h}^2\bigr]\notag\\
&\quad-\bigl[Ch^2\|\nabla d_t\theta^n\|_{0,2,h}^2+Ch^2\|\nabla\widetilde\Delta_h^{-1}d_t\theta^n\|_{0,2,h}^2\notag\\
&\quad+\frac{\epsilon^4}{16}\|\nabla\theta\|_{0,2,h}^2+\|\nabla\widetilde\Delta^{-1}_h\theta^n\|_{0,2,h}^2+Ch^2\|\theta^n\|_{0,2,h}^2\bigr]\notag\\
&\ge\bigl[\frac{k}{2}\|\nabla \widetilde\Delta^{-1}_hd_t\theta^n\|_{0,2,h}^2+\frac{1}{2k}\|\nabla\widetilde\Delta^{-1}_h\theta^n\|_{0,2,h}^2-\frac{1}{2k}\|\nabla\widetilde\Delta^{-1}_h\theta^{n-1}\|_{0,2,h}^2\bigr]\notag\\
&\quad-\bigl[\frac{\epsilon^4 k^2}{32}\|\nabla d_t\theta^n\|_{0,2,h}^2+\frac{\epsilon^4}{8}\|\nabla\theta\|_{0,2,h}^2+\|\nabla\widetilde\Delta^{-1}_h\theta^n\|_{0,2,h}^2\bigr]\notag\\
&\ge\bigl[\frac{k}{2}\|\nabla \widetilde\Delta^{-1}_hd_t\theta^n\|_{0,2,h}^2+\frac{1}{2k}\|\nabla\widetilde\Delta^{-1}_h\theta^n\|_{0,2,h}^2-\frac{1}{2k}\|\nabla\widetilde\Delta^{-1}_h\theta^{n-1}\|_{0,2,h}^2\bigr]\notag\\
&\quad-\bigl[\frac{\epsilon^4}{16}(\|\nabla\theta^n\|_{0,2,h}^2+\|\nabla\theta^{n-1}\|_{0,2,h}^2)+\frac{\epsilon^4}{8}\|\nabla\theta\|_{0,2,h}^2+\|\nabla\widetilde\Delta^{-1}_h\theta^n\|_{0,2,h}^2\bigr]\notag.
\end{align}


When $h\le \frac{C}{\beta}\epsilon^2$, the third term on the left-hand side of \eqref{eq20170507_2} can be bounded by
\begin{align}\label{eq20170716_2}
L_3&\le C\beta^2h^2\|\theta^n\|_{1,2,h}^2+\epsilon^2\|\theta^n\|_{0,2,h}^2\\
&\le \frac{\epsilon^4}{8}|\theta^n|_{1,2,h}^2+C\|\nabla\widetilde\Delta^{-1}_h\theta^n\|_{0,2,h}^2.\notag
\end{align}

When $h\le C\epsilon^{4\gamma_1+4}$ and $h\le C\epsilon^4$, the fourth term on the left-hand side of \eqref{eq20170507_2} can be bounded by
\begin{align}\label{eq20170716_4}
L_4&=\frac{1}{\epsilon}(\nabla f(P_hu(t_n))-\nabla f(u_h^n)),\nabla \Delta_h^{-1}\theta^n-\nabla \widetilde\Delta_h^{-1}\theta^n)_h\\
&\quad+\frac{1}{\epsilon}(\nabla f(P_hu(t_n))-\nabla f(u_h^n)),-\nabla \Delta_h^{-1}\theta^n)_h\notag\\
&\ge -\frac{h}{\epsilon}\|\nabla\bigl(f'(P_hu)\theta^n-f''(P_hu)(\theta^n)^2+(\theta^n)^3\bigr)\|_{0,2,h}^2-Ch\|\theta^n\|_{1,2,h}^2\notag\\
&\quad+\frac{1}{\epsilon}(\Delta_h^{-1}\theta^n-\Delta^{-1}\theta^n,f(P_hu(t_n))- f(u_h^n)))_h+\frac{1}{\epsilon}(f(P_hu(t_n))- f(u_h^n),\theta^n)_h\notag\\
&\ge -[C\frac{h}{\epsilon^{4\gamma_1}}|\theta^n|_{1,2,h}^2+Ch|\theta^n|_{1,2,h}^2]-[C\frac{h}{\epsilon^{2}}|\theta^n|_{0,2,h}^2+C\frac{h}{\epsilon^{4\gamma_1}}|\theta^n|_{0,2,h}^2]\notag\\
&\quad+\bigg[\frac{1}{\epsilon}(f'(P_hu(t_n)))\theta^n,\theta^n)_h-\frac{3}{\epsilon}P_hu(t_n)((\theta^n)^2,\theta^n)_h+\frac{1}{\epsilon}((\theta^n)^3,\theta^n)_h\bigg]\notag\\
&\ge -\frac{\epsilon^4}{8}|\theta^n|_{1,2,h}^2-C\|\nabla\widetilde{\Delta}_h^{-1}\theta^n\|_{0,2,h}\notag\\
&\quad+\bigg[\frac{1}{\epsilon}(f'(P_hu(t_n)))\theta^n,\theta^n)_h-\frac{3}{\epsilon}P_hu(t_n)((\theta^n)^2,\theta^n)_h+\frac{1}{\epsilon}((\theta^n)^3,\theta^n)_h\bigg]\notag.
\end{align}


For the second term inside the brackets and on the right-hand side of \eqref{eq20170716_4}, we appeal to Remark \ref{rem20180421_1}, the discrete energy law and the following Gagliardo-Nirenberg inequality \cite{adams2003sobolev}. Then for any $K\in \cT_h$, we have
\begin{align}\label{eq20170716_5_add}
\|\theta^n\|_{L^3(K)}^3&\leq C\Bigl( \|\nab\theta^n\|_{L^2(K)} 
\|\theta^n\|_{L^2(K)}^{2} +\|\theta^n\|_{L^2(K)}^3 \Bigr)\\
&\le \frac{\epsilon^5}{32C}\|\nab\theta^n\|_{L^2(K)}^2+\frac{C}{\epsilon^{\frac{\sigma_1}{2}+\frac{11}{4}}}\|\theta^n\|_{L^2(K)}^3\notag,\\
&\le \frac{\epsilon^5}{32C}\|\nab\theta^n\|_{L^2(K)}^2+\bigl[\frac{\epsilon^5}{32C}\|\nab\theta^n\|_{L^2(K)}^2+\frac{C}{\epsilon^{2\sigma_1+11}}\|\nabla\Delta^{-1}\theta^n\|_{L^2(K)}^3\bigr]\notag.
\end{align}

When $h\le C\epsilon^{\sigma_1+2}$, the second term inside the brackets and on the right-hand side of \eqref{eq20170716_4} can be bounded by
\begin{align}\label{eq20171109_3}
&\frac{3}{\epsilon}P_hu(t_n)((\theta^n)^2,\theta^n)_h\\
&\quad\le\frac{\epsilon^4}{16}\|\nab\theta^n\|_{0,2,h}^2+\frac{C}{\epsilon^{2\sigma_1+12}}\|\nabla\Delta^{-1}\theta^n\|_{0,2,h}^3\notag\\
&\quad\le\frac{\epsilon^4}{16}\|\nab\theta^n\|_{0,2,h}^2+\frac{C}{\epsilon^{2\sigma_1+12}}\|\nabla\widetilde\Delta^{-1}\theta^n\|_{0,2,h}^3+\frac{C}{\epsilon^{2\sigma_1+12}}h^3\|\theta^n\|_{0,2,h}^3\notag\\
&\quad\le\frac{\epsilon^4}{8}\|\nab\theta^n\|_{0,2,h}^2+\frac{C}{\epsilon^{2\sigma_1+12}}\|\nabla\widetilde\Delta^{-1}\theta^n\|_{0,2,h}^3\notag.
\end{align}

The fifth term on the left-hand side of \eqref{eq20170507_2} can be bounded by
\begin{align}\label{eq20170716_5}
L_5&\ge -C\|R(u_{tt};n)\|_{H^{-1}}^2-|\nabla\widetilde\Delta_h^{-1}\theta^n|_{0,2,h}^2.
\end{align}

By the mean value theorem and \eqref{eq2.13_add}, the first term on the right-hand side of \eqref{eq20170507_2} can be bounded by
\begin{align}\label{eq20170716_6}
R_1&\le\frac{C}{\epsilon}(f''(\xi)(P_hu(t_n)-u(t_n)),-\nabla \widetilde\Delta_h^{-1}\theta^n)_h\\
&\le\frac{C}{\epsilon^2}\|P_hu(t_n)-u(t_n)\|_{0,2,h}^2+\|\nabla \widetilde\Delta_h^{-1}\theta^n\|_{0,2,h}^2\notag\\
&\le C\epsilon^{-\max\{2\sigma_1+7,2\sigma_3+4\}}h^4+\|\nabla \widetilde\Delta_h^{-1}\theta^n\|_{0,2,h}^2.\notag
\end{align}

By the Poincar$\acute{e}$ inequality and the bounds of $I_1$ in \cite{wu2018analysis}, the summation of the second term on the right-hand side of \eqref{eq20170507_2} can be written as

\begin{align}\label{eq20170716_7}
&\sum_{n=1}^{\ell}\bigl[\sum_{E\in\mathcal{E}_h}(\frac{\partial\Delta^{-1}(-d_t\rho^n+\alpha\rho^n)}{\partial n},-\widetilde\Delta_h^{-1}\theta^n)_E-(\nabla\Delta^{-1}(-d_t\rho^n+\alpha\rho^n),-\nabla\widetilde\Delta_h^{-1}\theta^n)_h\bigr]\\
&\le C\sum_{n=1}^{\ell}\|-d_t\rho^n+\alpha\rho^n\|_{H^{-1}}^2+C\sum_{n=1}^{\ell}\|\nabla\widetilde\Delta_h^{-1}\theta^n\|_{0,2,h}^2\notag\\
&\le C\tilde\rho_0(\epsilon)\frac{|\ln h|h^2}{k}+C\sum_{n=1}^{\ell}\|\nabla\widetilde\Delta_h^{-1}\theta^n\|_{0,2,h}^2,\notag
\end{align}
where Lemma 2.3 in \cite{elliott1989nonconforming} is used in the first inequality and 
\begin{align*}
\tilde\rho_0(\epsilon)&:=\epsilon^4\rho_3(\epsilon)+\epsilon^{-6}\rho_4(\epsilon)+\rho_5(\epsilon),\\
\rho_3(\epsilon)&:=\epsilon^{-\max\{2\sigma_1+\frac{13}{2},2\sigma_3+\frac{7}{2},2\sigma_2+4,2\sigma_4\}
  - \max\{2\sigma_1+5, 2\sigma_3+2\} - 2}+ \epsilon^{-\max\{\sigma_1+\frac52,\sigma_3+1\}-2}\rho_0(\epsilon) +
\epsilon^{-2\sigma_6+1}, \\ 
\rho_4(\epsilon) &:=
\epsilon^{-\max\{2\sigma_1+\frac{13}{2},2\sigma_3+\frac{7}{2},2\sigma_2+4,2\sigma_4\}
+4}, \\
\rho_5(\epsilon) &:=
\epsilon^{-2\max\{2\sigma_1+\frac{13}{2},2\sigma_3+\frac{7}{2},2\sigma_2+4,2\sigma_4\}
+2}.
\end{align*}


Using Remark \ref{rem20180421_1}, then under the mesh constraint $h\le Ck$ and $h\le C\epsilon^4$, the third term on the right-hand side of \eqref{eq20170507_2} can be bounded by
\begin{align}\label{eq20171106_12}
R_3&=\sum_{E\in\mathcal{E}_h}(\frac{\partial\Delta^{-1}(\Delta^{-1}d_t\theta^n-\Delta_h^{-1}d_t\theta^n)}{\partial n},\widetilde\Delta_h^{-1}\theta^n)_E\\
&\quad-(\nabla\Delta^{-1}(\Delta^{-1}d_t\theta^n-\Delta_h^{-1}d_t\theta^n),\nabla\widetilde\Delta_h^{-1}\theta^n)_h\notag\\ 
\le&Ch\|d_t\theta^n\|_{L^2}\|\nabla\widetilde\Delta_h^{-1}\theta^n\|_{L^2}\notag\\
\le&(Ch^3\|\nabla d_t\theta^n\|_{L^2}^2+\frac{h}{4C}\|\nabla\widetilde\Delta_h^{-1}d_t\theta^n\|_{L^2}^2)+C\|\nabla\widetilde\Delta_h^{-1}\theta^n\|_{L^2}^2\notag\\
\le&\frac{\epsilon^4k^2}{32}\|\nabla d_t\theta^n\|_{L^2}^2+\frac{k}{4}\|\nabla\widetilde\Delta_h^{-1}d_t\theta^n\|_{L^2}^2+C\|\nabla\widetilde\Delta_h^{-1}\theta^n\|_{L^2}^2\notag\\
\le&\frac{\epsilon^4}{16}(\|\nabla\theta^n\|_{L^2}^2+\|\nabla\theta^{n-1}\|_{L^2}^2)+\frac{k}{4}\|\nabla\widetilde\Delta_h^{-1}d_t\theta^n\|_{L^2}^2+C\|\nabla\widetilde\Delta_h^{-1}\theta^n\|_{L^2}^2\notag,
\end{align}
where Lemma 2.3 in \cite{elliott1989nonconforming} and the inverse inequality are used in the first inequality.

Combining \eqref{eq20170714_1} to \eqref{eq20171106_12}, we have
\begin{align}\label{eq20170808_1}
&\frac{1}{2}\|\nabla\widetilde \Delta_h^{-1}\theta^n\|_{0,2,h}^2-\frac{1}{2}\|\nabla \widetilde\Delta^{-1}_h\theta^{n-1}\|_{0,2,h}^2+\frac{k^2}{4}\|\nabla\widetilde \Delta_h^{-1}d_t\theta^n\|_{0,2,h}^2\\
&\quad+k(\epsilon-\frac{7\epsilon^4}{8})(\nabla\theta^n,\nabla\theta^n)_h+
\frac{k}{\epsilon}(f'(P_hu(t_n)))\theta^n,\theta^n)_h+\frac{k}{\epsilon}\|\theta^n\|_{0,4,h}^4\notag\\
&\le Ck(-d_t\rho^n+\alpha\rho^n,-\widetilde\Delta_h^{-1}\theta^n)+Ck\|R(u_{tt};n)\|_{H^{-1}}^2+\frac{Ck}{\epsilon^{2\sigma_1+12}}\|\nabla\widetilde\Delta_h^{-1}\theta^n\|_{0,2,h}^3\notag\\
&\quad+C\epsilon^{-\max\{2\sigma_1+7,2\sigma_3+4\}}h^4+Ck\|\nabla\widetilde\Delta_h^{-1}\theta^n\|_{0,2,h}^2.\notag
\end{align}


Taking the summation for $n$ from $1$ to $\ell$, equation \eqref{eq20170808_1} can be changed into
\begin{align}\label{eq20170808_3}
&\frac{1}{2}\|\nabla \widetilde\Delta_h^{-1}\theta^{\ell}\|_{0,2,h}^2+\frac{k^2}{4}\sum_{n=1}^\ell\|\nabla \widetilde\Delta_h^{-1}d_t\theta^n\|_{0,2,h}^2\\
&\quad+k\sum_{n=1}^\ell\bigl[(\epsilon-\epsilon^4)(\nabla\theta^n,\nabla\theta^n)_h+
\frac{1}{\epsilon}(f'(P_hu(t_n)))\theta^n,\theta^n)_h\bigr]\notag\\
&\quad+\frac{\epsilon^4k}{8}\sum_{n=1}^\ell(\nabla\theta^n,\nabla\theta^n)_h+\frac{k}{\epsilon}\sum_{n=1}^\ell\|\theta^n\|_{0,4,h}^4\notag\\
&\le C\tilde\rho_1(\epsilon)k^2+\frac{Ck}{\epsilon^{2\sigma_1+12}}\sum_{n=1}^\ell\|\nabla\widetilde\Delta_h^{-1}\theta^n\|_{0,2,h}^3+Ck\sum_{n=1}^\ell\|\nabla\widetilde\Delta_h^{-1}\theta^n\|_{0,2,h}^2\notag\\
&\quad+C\tilde\rho_0(\epsilon)|\ln h|h^2\notag.
\end{align}

Using the generalized coercivity result in Theorem \ref{thm3.7_add}, we obtain when $h\le C\epsilon^2$, we have
\begin{align}\label{eq20170808_4}
&\frac{1}{2}\|\nabla \widetilde\Delta_h^{-1}\theta^{\ell}\|_{0,2,h}^2+\frac{k^2}{4}\sum_{n=1}^\ell\|\nabla \widetilde\Delta_h^{-1}d_t\theta^n\|_{0,2,h}^2\\
&\quad+\frac{\epsilon^4k}{8}\sum_{n=1}^\ell(\nabla\theta^n,\nabla\theta^n)_h+\frac{k}{\epsilon}\sum_{n=1}^\ell\|\theta^n\|_{0,4,h}^4\notag\\
&\le C\tilde\rho_1(\epsilon)k^2+\frac{Ck}{\epsilon^{2\sigma_1+12}}\sum_{n=1}^\ell\|\nabla\widetilde\Delta_h^{-1}\theta^n\|_{0,2,h}^3+Ck\sum_{n=1}^\ell\|\nabla\widetilde\Delta_h^{-1}\theta^n\|_{0,2,h}^2\notag\\
&\quad+C\tilde\rho_0(\epsilon)|\ln h|h^2+Ck\sum_{n=1}^\ell\|\nabla\Delta^{-1}\theta^n\|_{L^2}^2+Ch^2\epsilon^{-2\gamma_2-4}\notag\\
&\le C\tilde\rho_1(\epsilon)k^2+\frac{Ck}{\epsilon^{2\sigma_1+12}}\sum_{n=1}^\ell\|\nabla\widetilde\Delta_h^{-1}\theta^n\|_{0,2,h}^3+Ck\sum_{n=1}^\ell\|\nabla\widetilde\Delta_h^{-1}\theta^n\|_{0,2,h}^2\notag\\
&\quad+\frac{\epsilon^4k}{16}\sum_{n=1}^\ell(\nabla\theta^n,\nabla\theta^n)_h+C\tilde\rho_0(\epsilon)|\ln h|h^2\notag.
\end{align}

By the discrete energy law and Theorem \ref{thm20180517_1}, when $k\le C\epsilon^{3\sigma_1+13}$, we have
\begin{align}\label{eq20170808_5}
&\frac{1}{4}\|\nabla \widetilde\Delta_h^{-1}\theta^{\ell}\|_{0,2,h}^2+\frac{k^2}{4}\sum_{n=1}^\ell\|\nabla \widetilde\Delta_h^{-1}d_t\theta^n\|_{0,2,h}^2\\
&\quad+\frac{\epsilon^4k}{16}\sum_{n=1}^\ell(\nabla\theta^n,\nabla\theta^n)_h+\frac{k}{\epsilon}\sum_{n=1}^\ell\|\theta^n\|_{0,4,h}^4\notag\\
&\le C\tilde\rho_1(\epsilon)k^2+\frac{Ck}{\epsilon^{2\sigma_1+12}}\sum_{n=1}^{\ell-1}\|\nabla\widetilde\Delta_h^{-1}\theta^n\|_{0,2,h}^3+Ck\sum_{n=1}^{\ell-1}\|\nabla\widetilde\Delta_h^{-1}\theta^n\|_{0,2,h}^2\notag\\
&\quad+C\tilde\rho_0(\epsilon)|\ln h|h^2\notag.
\end{align}

Let $d_{\ell}\geq 0$ be the slack variable such that
\begin{align}\label{eq20170808_6}
&\frac{1}{4}\|\nabla \widetilde\Delta_h^{-1}\theta^{\ell}\|_{0,2,h}^2+\frac{k^2}{4}\sum_{n=1}^\ell\|\nabla \widetilde\Delta_h^{-1}d_t\theta^n\|_{0,2,h}^2\\
&\quad+\frac{\epsilon^4k}{16}\sum_{n=1}^\ell(\nabla\theta^n,\nabla\theta^n)_h+\frac{k}{\epsilon}\sum_{n=1}^\ell\|\theta^n\|_{0,4,h}^4+d_{\ell}\notag\\
&= \frac{Ck}{\epsilon^{2\sigma_1+12}}\sum_{n=1}^{\ell-1}\|\nabla\widetilde\Delta_h^{-1}\theta^n\|_{0,2,h}^3+Ck\sum_{n=1}^{\ell-1}\|\nabla\widetilde\Delta_h^{-1}\theta^n\|_{0,2,h}^2\notag\\
&\quad+C\tilde\rho_1(\epsilon)k^2+C\tilde\rho_0(\epsilon)|\ln h|h^2\notag.
\end{align}
and define for $\ell\geq1$
\begin{align}\label{eq3.55}
S_{\ell+1}:&= \frac{1}{4}\|\nabla \widetilde\Delta_h^{-1}\theta^{\ell}\|_{0,2,h}^2+\frac{k^2}{4}\sum_{n=1}^\ell\|\nabla \widetilde\Delta_h^{-1}d_t\theta^n\|_{0,2,h}^2\\
&\quad+\frac{\epsilon^4k}{16}\sum_{n=1}^\ell(\nabla\theta^n,\nabla\theta^n)_h+\frac{k}{\epsilon}\sum_{n=1}^\ell\|\theta^n\|_{0,4,h}^4+d_{\ell}\notag\\
S_{1}:&=C\tilde\rho_1(\epsilon)k^2+C\tilde\rho_0(\epsilon)|\ln h|h^2,
\end{align}
then we have
\begin{equation}\label{eq3.56}
S_{\ell+1}-S_{\ell}\leq CkS_{\ell}+\frac{Ck}{\epsilon^{2\sigma_1+12}}S_{\ell}^{\frac32}\qquad\text{for}\ \ell\geq1.
\end{equation}

Applying Lemma \ref{lem2.1} to $\{S_\ell\}_{\ell\geq 1}$ defined above, 
we obtain for $\ell\geq1$
\begin{equation}\label{eq3.58}
S_{\ell}\leq a^{-1}_{\ell}\Bigg\{S^{-\frac12}_{1}-\frac{Ck}{\epsilon^{2\sigma_1+12}}\sum_{s=1}^{\ell-1} a^{-\frac12}_{s+1}\Bigg\}^{-2}
\end{equation}
provided that
\begin{equation}\label{eq3.59}
S^{-\frac12}_{1}-\frac{Ck}{\epsilon^{2\sigma_1+12}}\sum_{s=1}^{\ell-1} a^{-\frac15}_{s+1}>0.
\end{equation}
We note that $a_s\, (1\leq s\leq \ell)$ are all bounded as $k\rightarrow0$, 
therefore, \eqref{eq3.59} holds under the mesh constraint stated in the theorem. Then it follows from \eqref{eq3.58} and \eqref{eq3.59} that 
\begin{equation*}
S_{\ell}\leq 2a_\ell^{-1} S_1
\leq C(\tilde\rho_0(\epsilon)|\ln h|h^2+\tilde\rho_1(\epsilon)k^2).
\end{equation*} 

Then the theorem is proved. Notice the mesh restrictions are stringent theoretically, and numerically they are much better.
\end{proof}

Next we gave a Corollary based on Theorem \ref{thm20171007_1}, Lemmas \ref{lem20170710_1_add}--\ref{lem20171106_2}, and the triangle inequality.
\begin{corollary}
Assume the mesh constraints in Theorem \ref{thm20171007_1} hold, then the following estimates hold
\begin{align*}
\|\nabla \hat\Delta_h^{-1}\theta^\ell\|_{0,2,h}^2&\leq C(\tilde\rho_0(\epsilon)|\ln h|h^2+\tilde\rho_1(\epsilon)k^2),\\
\|\nabla \Delta_h^{-1}\theta^\ell\|_{0,2,h}^2&\leq C(\tilde\rho_0(\epsilon)|\ln h|h^2+\tilde\rho_1(\epsilon)k^2),\\
\|\nabla \underline\Delta_h^{-1}\theta^\ell\|_{0,2,h}^2&\leq C(\tilde\rho_0(\epsilon)|\ln h|h^2+\tilde\rho_1(\epsilon)k^2),\\
\|\nabla \Delta^{-1}\theta^\ell\|_{0,2,h}^2&\leq C(\tilde\rho_0(\epsilon)|\ln h|h^2+\tilde\rho_1(\epsilon)k^2).\\
\end{align*}
\end{corollary}

\begin{remark}
\begin{enumerate}
\item All mesh restrictions have been incorporated into Theorem \ref{thm20171007_1}. For example, $h\le C\epsilon^4$ can be incorporated into $h\le C\epsilon^{4\gamma_1+4}$ since $\gamma_1\ge0$, $k\ge C\frac{h^4}{\epsilon^{4+4\gamma_1+2\sigma_1}}(\ln\,\frac1h)^2$ in Theorem \ref{thm20171006_3} can be incorporated into $k\ge C\frac{h^2}{\epsilon^{4\gamma_1+3}}$ and $h^2(\ln\,\frac1h)^2\le\epsilon^{2\sigma_1+1}$, and so on.

\item If $v_h=-\hat\Delta_h^{-1}\theta^n$, instead of $v_h=-\widetilde\Delta_h^{-1}\theta^n$, is chosen as the test function in \eqref{eq20170504_17}, the error estimate can not be obtained due to other terms in the definition of $-\hat\Delta_h^{-1}$.
\end{enumerate}
\end{remark}

\section{Numerical Experiments}\label{sec4}
In this section, we present two numerical tests to gauge the performance 
of the Morley element approximation. The fully implicit scheme and the square domain $\Omega=[-1,1]^2$ are used in both tests. The degrees of freedom (DOF) are compared for quadratic mixed discontinuous Galerkin method (MDG), $C^1$ conforming Argyris element, $C^1$ conforming Hsieh-Clough-Tocher (HCT) macro element, and Morley element in Table \ref{tab1} below. From the angle of degrees of freedom, Morley element method is supposed to be very efficient.

\begin{table}[H]
\large
\begin{center}
\begin{tabular}{cccccccccc}
\hline
& MDG & Argyris & HCT & Morley \\ \hline
$h=0.4$ & 1200 & 526 & 343 & 221\\ 
$h=0.2$ &4800  & 1946 & 1283 & 841  \\ 
$h=0.1$ & 19200 & 7486 & 4963  & 3281\\ 
$h=0.05$ & 76800 & 29366 & 19523 & 12961 \\ 
$h=0.025$ & 307200 & 116326 & 77443 & 51521 \\ \hline
\end{tabular}
\smallskip
\caption{Approximate number of DOF using MDG, Argyris, HCT and Morley elements.} 
\label{tab1} 
\end{center}
\end{table}

Next, two numerical tests are presented to numerically check the discrete maximum principle, which is not known theoretically. See \cite{wu2018analysis} for evolutions of the zero-level sets of the Cahn-Hilliard equation using the Morley elements based on more different initial conditions.

$\mathbf{Test\, 1.}$ Consider the Cahn-Hilliard equations (\ref{eq20170504_1})-(\ref{eq20170504_5}) with the following initial condition:
\begin{equation}\label{eq20180716_1}
 u_0(x)=\tanh\Bigl(\frac{d_0(x)}{\sqrt{2}\eps}\Bigr),
\end{equation}
where $d_0(x)=\sqrt{x_1^2+x_2^2}-0.5$, which is the signed distance from any point to the circle $x_1^2+x_2^2=0.5^2$. Note that $u_0$ has the desired form as stated in Lemma \ref{lem3.4}.

Figure \ref{fig1} plots the zero-level set of this initial condition and $L^\infty$ bound $|u_h^n|_{L^{\infty}}$. We can observe that $|u_h^n|_{L^{\infty}}\le1$, which numerically verifies the assumption \eqref{eq20180425_5}. In this test, the interaction length $\epsilon=0.05$, the space size $h=0.04$ and the time step size $k=0.0001$.

\begin{figure}[tbh]
   \centering
   \includegraphics[width=2.1in,,height=1.8in]{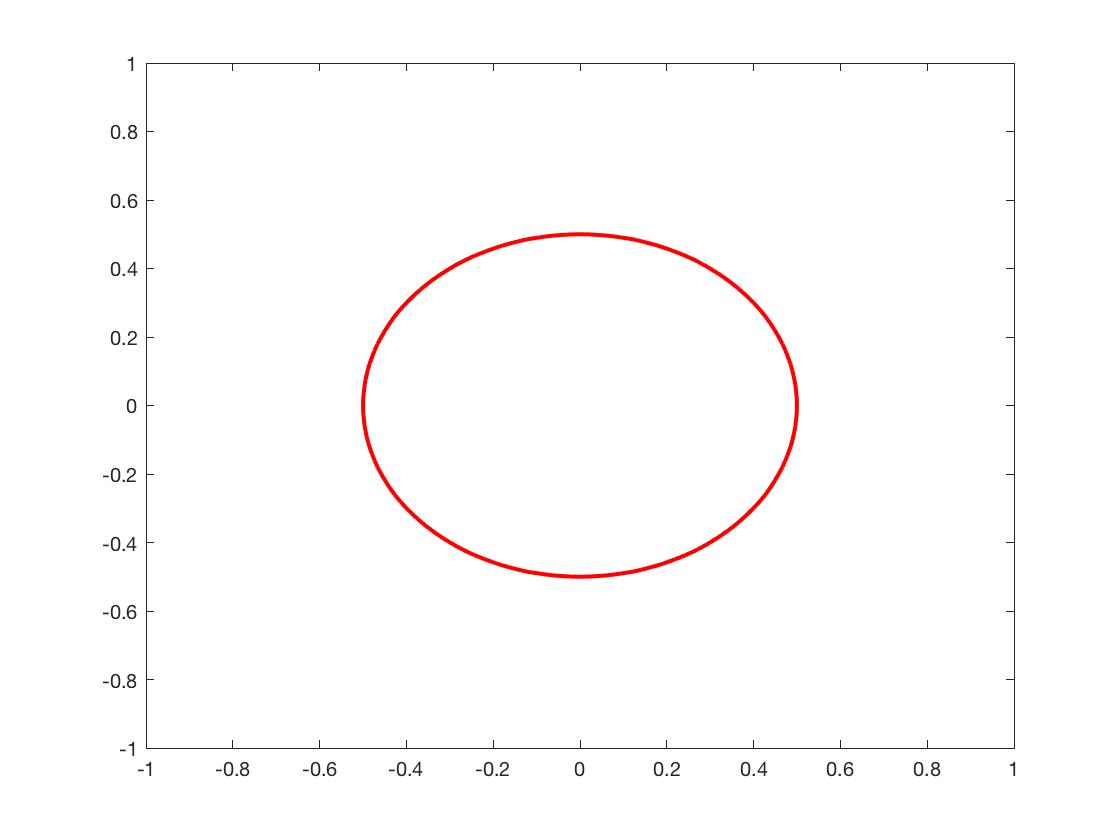} 
   \includegraphics[width=2.1in,,height=1.8in]{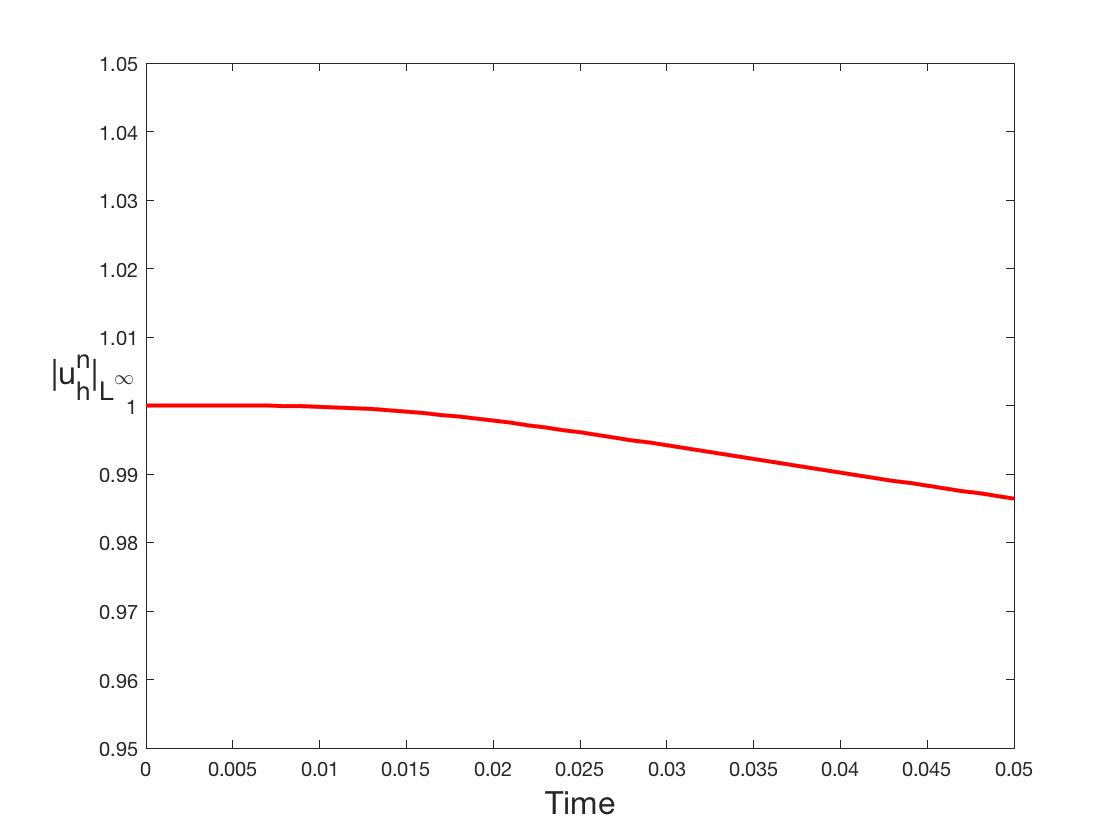}   
   \caption{The zero-level set of the initial condition (left) and the $|u_h^n|_{L^{\infty}}$ bound at different time points (right). In this test, $\epsilon=0.05,\,h=0.04,\,k=0.0001$.} \label{fig1}
\end{figure}
 
\medskip
$\mathbf{Test\, 2.}$ Consider the Cahn-Hilliard equations (\ref{eq20170504_1})-(\ref{eq20170504_5}) with the following initial condition:
\begin{equation*}
u_0(x)=\tanh\Bigl(\frac{1}{\sqrt{2}\eps}
\bigl(\min\bigl\{\sqrt{(x_1+0.3)^2+x_2^2}-0.3,\sqrt{(x_1-0.3)^2+x_2^2}-0.25\bigr\}\bigr)\Bigr).
\end{equation*}
Note that $u_0$ can be written in the form given in \eqref{eq20180716_1}
with $d_0(x)$ being the signed distance function to the initial curve. We note that $u_0$ does not have the desired form as stated in Lemma \ref{lem3.4}.

Figure \ref{fig2} plots the zero-level set of this initial condition and $L^\infty$ bound $|u_h^n|_{L^{\infty}}$. We can observe that $|u_h^n|_{L^{\infty}}\le 1$, which numerically verifies the assumption \eqref{eq20180425_5}. In this test, the interaction length $\epsilon=0.025$, the space size $h=0.02$, and the time step size $k=0.0001$.
\begin{figure}[tbh]
   \centering
   \includegraphics[width=2.1in,,height=1.8in]{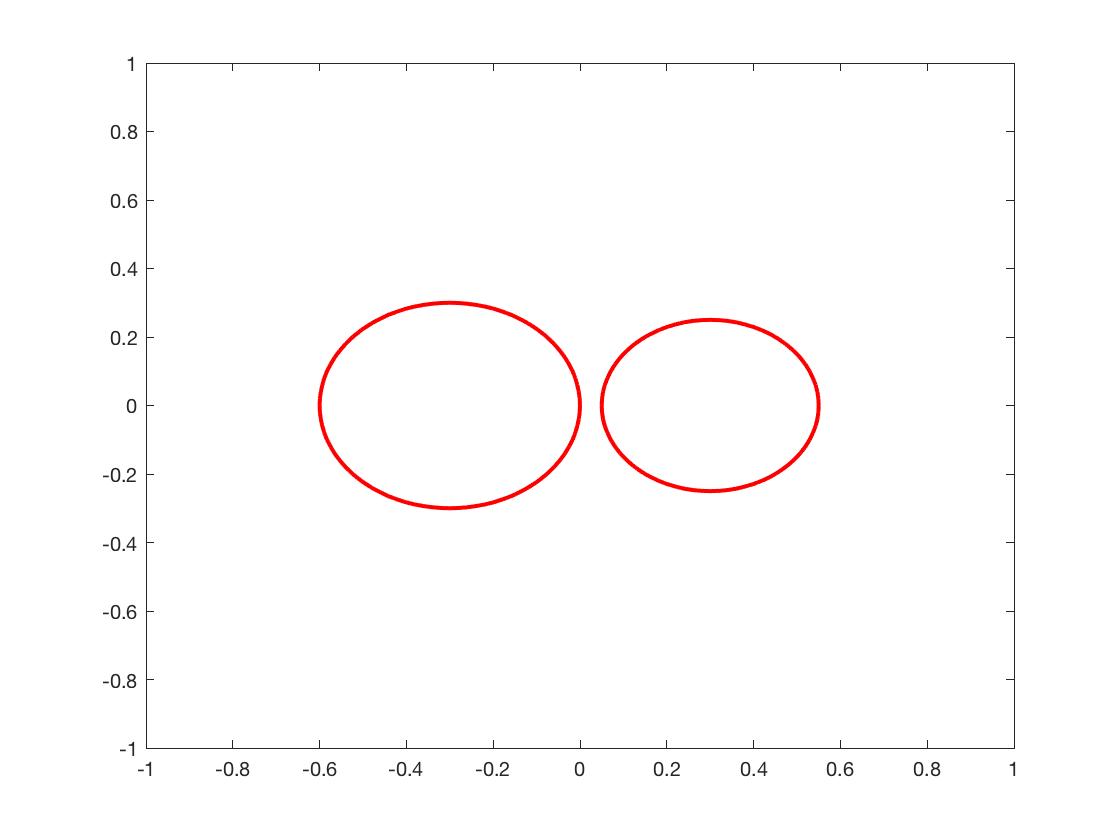} 
   \includegraphics[width=2.1in,,height=1.8in]{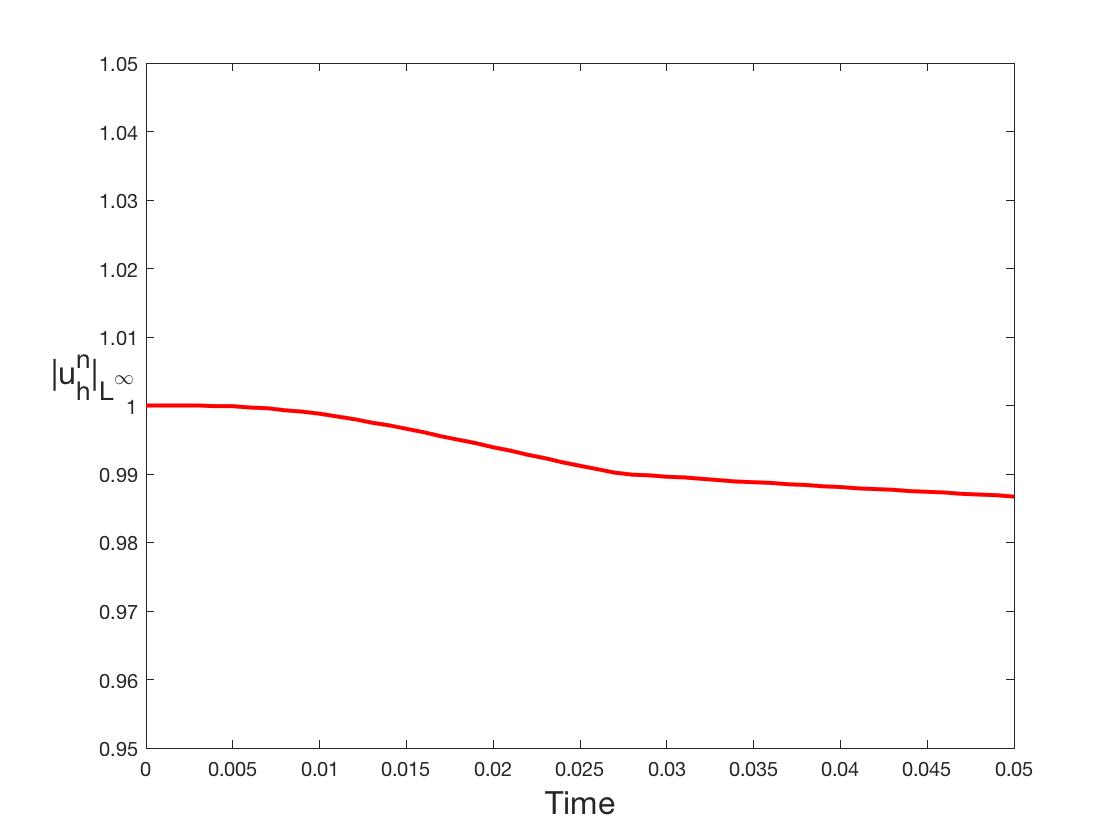}   
   \caption{The zero-level set of the initial condition (left) and the $|u_h^n|_{L^{\infty}}$ bound at different time points (right). In this test, $\epsilon=0.025,\,h=0.02,\,k=0.0001$.} \label{fig2}
\end{figure}

\section*{Acknowledgements}
The author Yukun Li highly thanks Professor Xiaobing Feng in the University of Tennessee at Knoxville for his valuable suggestions during the whole process of preparation of this manuscript, and Dr. Shuonan Wu in the Peking University for proofreading this manuscript carefully and giving lots of useful suggestions.




\end{document}